\documentclass[a4paper,11pt]{article}

\usepackage[english]{babel}  
\usepackage[utf8]{inputenc}
\usepackage[totalwidth=16.5cm,totalheight=23.5cm]{geometry}

\usepackage[overload]{empheq}
\usepackage{tabularx,graphicx,changepage}
\usepackage{amsmath,amssymb,amsthm,mathtools,bm}
\usepackage{hyperref}
\usepackage{comment}

\newtheorem{Proposition}{Proposition}
\newtheorem{Theorem}{Theorem}
\newtheorem{Lemma}{Lemma}	
\newtheorem{Corollary}{Corollary}
\theoremstyle{definition}
\newtheorem{Definition}{Definition}
\DeclareMathOperator{\Ric}{Ric}
\DeclareMathOperator{\PSL}{PSL}
\newcommand{\R}{\mathbb{R}}

\newcommand{\from}{:}
\newcommand{\Oc}{\mathcal{O}}

\newcommand{\gs}{\sigma}
\newcommand{\gS}{\Sigma}

\newcommand{\gr}{\rho}
\newcommand{\gO}{\Omega}

\newcommand{\gG}{\Gamma}
\newcommand{\g}{\gamma}
\newcommand{\dg}{\dot{\gamma}}
\newcommand{\Sc}{\mathcal{S}}

\newcommand{\hConf}{[h]}

\begin{document}
	\pagestyle{plain}
	\title{An ambient approach to conformal geodesics}
	\author{Joel Fine\footnote{Joel.Fine@ulb.ac.be},
	\, Yannick Herfray\footnote{Yannick.Herfray@ulb.ac.be}\\ 
		{\small \it D\'epartment de Math\'ematique,
			Universit\'e Libre de Bruxelles, } \\ {\small \it CP 218, Boulevard du Triomphe, B-1050 Bruxelles, Belgique.}}
	\maketitle
	\begin{abstract}\noindent
\noindent
Conformal geodesics are distinguished curves on a conformal manifold, loosely analogous to geodesics of Riemannian geometry. One definition of them is as solutions to a third order differential equation determined by the conformal structure. There is an alternative description via the tractor calculus. In this article we give a third description using ideas from holography. A conformal $n$-manifold $X$ can be seen (formally at least) as the asymptotic boundary of a Poincar\'e--Einstein $(n+1)$-manifold $Y$. We show that any curve $\gamma$ in $X$ has a uniquely determined extension to a surface $\Sigma_\gamma$ in $Y$, which which we call the \emph{ambient surface of $\gamma$}. This surface meets the boundary $X$ in right angles along $\gamma$ and is singled out by the requirement that it it be a critical point of renormalised area. The conformal geometry of $\gamma$ is encoded in the Riemannian geometry of $\Sigma_\gamma$. In particular, $\gamma$ is a conformal geodesic precisely when $\Sigma_\gamma$ is asymptotically totally geodesic, i.e.\ its second fundamental form vanishes to one order higher than expected. 

We also relate this construction to tractors and the ambient metric construction of Fefferman and Graham. In the $(n+2)$-dimensional ambient manifold, the ambient surface is a graph over the bundle of scales. The tractor calculus then identifies with the usual tensor calculus along this surface. This gives an alternative compact proof of our holographic characterisation of conformal geodesics.
	\end{abstract}
	\clearpage
	
	\section{Introduction}
	
	\subsection{The Fefferman--Graham approach to conformal geometry}
We begin with a review of the ambient metric approach to conformal geometry pioneered by Fefferman and Graham \cite{fefferman_conformal_1985,fefferman_ambient_2012}. More details are provided later in the article for the benefit of those unfamiliar with these constructions.
	
	Given a compact $n$-manifold $X$ and a conformal class $[h]$ of metrics with signature $(p,q)$ on $X$, Fefferman and Graham gave two different ways to extend $(X,[h])$ to genuine metrics solving two versions of Einstein's equations, one a $(n+1)$-manifold with metric of signature $(p+1,q)$ and  $\Ric(g)=-ng$, the other a $(n+2)$-manifold with metric of signature $(p+1,q+1)$ and $\Ric=0$. We focus in this introduction on the $(n+1)$-dimensional extension but in the main body of the article both extensions are used. From now on we focus on Riemannian signature $(p,q)=(n,0)$ but the results extend straightforwardly to other signatures.
	
	The $(n+1)$-dimensional manifold $Y$ is diffeomorphic to $X \times [0,1)$. Given a representative $h_0 \in[h]$ of the conformal structure on $X$, there exists a unique way to identify $Y$ and $X \times [0,1)$ under which $g$ has the form
	\begin{equation}\label{conformally-compact}
		g = \frac{d \rho^2 + h(\rho)}{\rho^2} 
	\end{equation}
	Here $\rho$ is the projection onto the factor $[0,1)$ and $h(\rho)$ is a path of metrics on $X$ with $h(0) = h_0$. More correctly, this is a \emph{formal} construction. Things are cleanest when $n$ is odd. In this case there is a formal power series expansion of $h$ in even powers of $\rho$, of the form 
	\begin{equation}
		h(\rho) = h_0 + h_2 \rho^2 + h_4 \rho^4 + \cdots 
		\label{FG-expansion}
	\end{equation}
	The metric is Einstein in the formal sense that $\Ric(g)+ng = \Oc(\rho^{-\infty})$. (The word ``formal'' here means no attention is paid to convergence; $\Oc(\rho^{-\infty})$ means all coefficients of the corresponding power series vanish.) The expansion~\eqref{FG-expansion} is canonically associated to $h_0$: the coefficients $h_{2k}$ are completely determined by the curvature of $h_0$ and its derivatives. Moreover, whilst different choices of representatives of $[h]$ lead to different expansions, the resulting metrics are \emph{isometric}, just not in a way which is compatible with the identifications $Y \cong X \times [0,1)$. It follows that the Riemannian manifold $(Y,g)$ is canonically associated to the conformal manifold $(X,[h])$
	
	When $n \geq 4$ is even, the analogous expansion can be carried out up to $\Oc(\rho^{n-1})$ but the next step is obstructed (at least in general). The result is a metric which has $\Ric(g)+ng = \Oc(\rho^n)$. (The case $n=2$ is exceptional; we momentarily ignore it here, but treat it carefully in the main body of the article.) No matter the parity of the dimension, to a conformal $n$-manifold $(X,[h])$ Fefferman and Graham produce a germ to order $n$ of a formal Riemannian Einstein metric $(Y,g)$, called the \emph{formal Poincar\'e--Einstein metric with conformal infinity $(X,[h])$}.
	
	Since $(Y,g)$ is canonically associated to $(X,[h])$ the Riemannian geometry of $Y$ can be used to understand the conformal geometry of $X$. This is the mathematical point of view on holography and it has been extremely successful, giving rise, for example, to new conformally invariant curvature tensors and differential operators (see for example \cite{fefferman_q-curvature_2002}). This article applies this approach to the study of \emph{curves} in $(X,[h])$.
	
	\subsection{The ambient surface of a curve $\gamma \subset X$}

	To any curve $\gamma \colon I \to X$, we canonically associate an \emph{ambient surface} $\Sigma \subset Y$  (see figure \ref{fig:AmbientSurface}). The surface meets the boundary at right angles in the curve $\gamma$. The area of such a surface is necessarily infinite, one can however define the so-called \emph{renormalized area} -which is finite. It was introduced in \cite{graham_conformal_1999} and has been the subject of much investigation in both the mathematics and physics literature (see, for example, \cite{alexakis_renormalized_2010, ryu_holographic_2006}). Given $h_0 \in [h]$ and the expansion~\eqref{FG-expansion}, one shows that the area $A(\epsilon)$ of $\Sigma$ contained in the region $\rho \geq \epsilon$ has an expansion in powers of $\epsilon$ of the form 
	\begin{equation}\label{area-expansion}
		A(\epsilon) = l(\gamma)\epsilon^{-1} + \mathcal{A} + \Oc(\epsilon)
	\end{equation}
	for some $\mathcal{A} \in \R$ (where $l(\gamma)$ is the length of $\gamma 
	\subset X$ with respect to $h_0$). Crucially, this constant term $\mathcal{A}$ \emph{does not depend on the choice of $h_0 \in [h]$}. This is the renormalised area of $\Sigma$. 
		
Our first result is the following.
	
	\begin{Theorem}\label{existence-uniqueness-ambient-surface}
		Let $(X,[h])$ be a conformal $n$-manifold with formal Poincar\'e--Einstein extension $(Y,g)$. Given any embedded curve $\gamma \subset X$ there exists a unique formal surface $\gS \subset Y$ which is a critical point for renormalised area and which meets $X$ in right angles along $\gamma$. 
	\end{Theorem}
	
	We call the distinguished surface of Theorem~\ref{existence-uniqueness-ambient-surface} the \emph{ambient surface associated to $\gamma$} and we denote it by $\Sigma_\gamma$.
	 We remark that, just as for the Fefferman--Graham construction, in general the ambient surface is purely formal. When $n$ is odd, it has a uniquely determined formal expansion to all orders. When $n$ is even, it exists to the same order as the Fefferman--Graham metric. For comments on the global existence question \\S\ref{global} below.
	 	\begin{figure}[h]
	 	\begin{center}
	 		\includegraphics[scale=0.5]{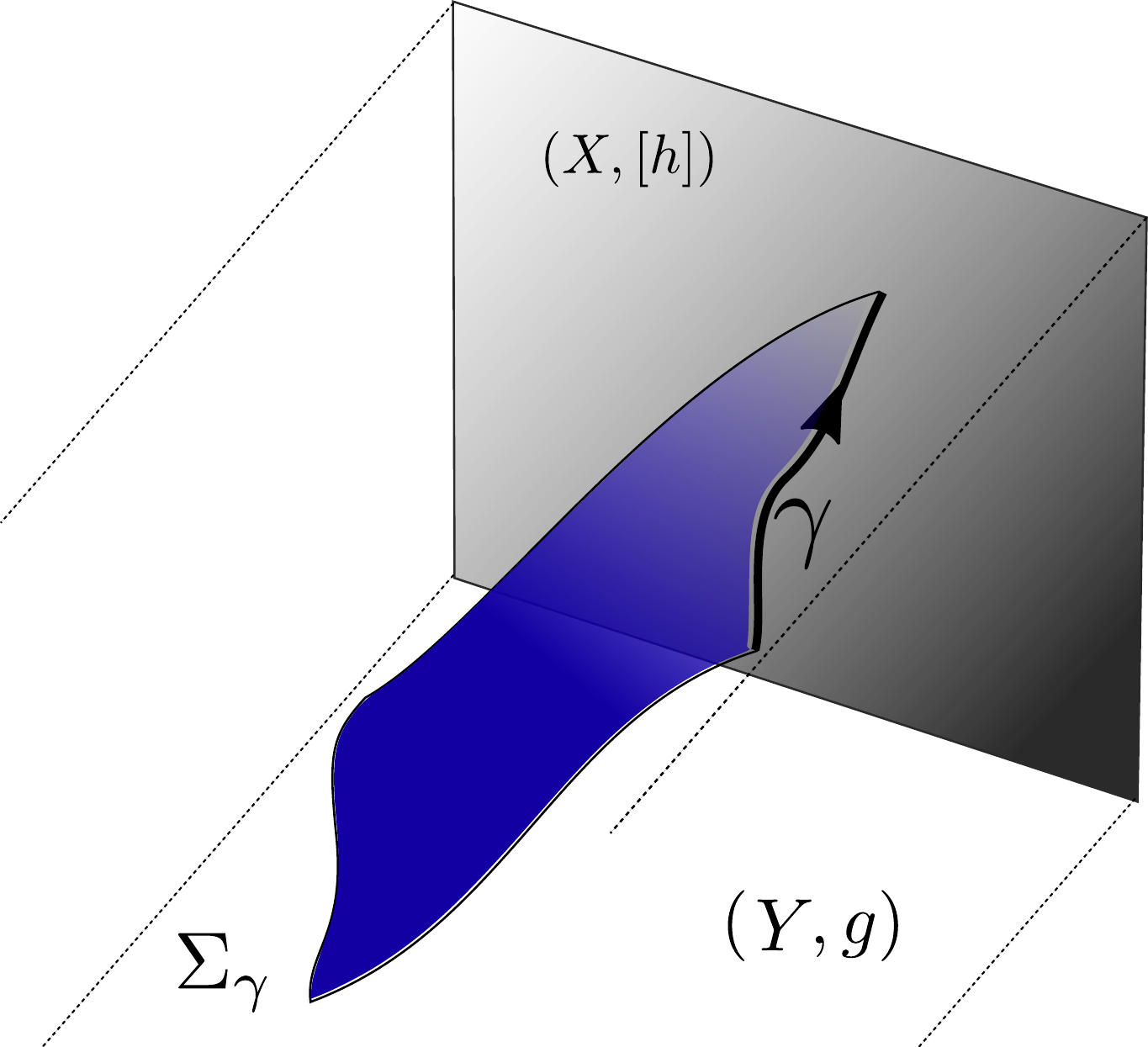}
	 	\end{center}
	 	\caption{The associated ambient surface $\Sigma_{\gamma} \subset Y$ of a conformal curve $\gamma \subset X$}
	 	\label{fig:AmbientSurface}
	 \end{figure}
	
	We say a few words about the proof of Theorem~\ref{existence-uniqueness-ambient-surface}. Surfaces which are critical points of renormalised area are automatically minimal. One can see this by considering variations of the surface which are compactly supported away from the boundary, and so which only affect the constant term in $A(\epsilon)$. Analysing the minimal surface equation, one sees that it has two indicial roots, $0$ and $3$. (This is done in \cite{alexakis_renormalized_2010} in dimension $n=2$, and in \cite{graham_higher-dimensional_2017} in higher dimensions.) This means that in the local existence problem, one can expect to prescribe the boundary value (Dirichlet data) corresponding to the root $0$, and a second piece of boundary data, corresponding to the root $3$ and which one can interpret as Neumann data. Once both Dirichlet and Neumann data are fixed, the solution is completely determined. (The simple analogue is that given a pair of functions $\phi,\psi$ on $S^1$ there is a unique harmonic function $h$ defined on a neighbourhood of $S^1 \subset \R^2$ with $h|_{S^1} = \phi$ and $\partial_n h = \psi$.)
	
	It turns out that the vanishing of the Neumann data is equivalent to the surface being critical for \emph{renormalised} area, under perturbations which are not compactly supported and actually move the boundary curve. (This was shown in \cite{alexakis_renormalized_2010} for dimension $n=2$, and in arbitrary dimensions is shown below.) This explains why one should expect a unique critical surface to emanate from any curve $\gamma$.
	
	We use this ambient surface $\Sigma_\gamma$ to study the conformal geometry of $\gamma \subset X$. A first point concerns the canonical parametrisations of $\gamma$. These are conformal analogues of parametrisation by arc length in Riemannian geometry, first introduced in \cite{bailey_conformal_1990}. Given $h \in [h]$, the curve $\gamma \colon I \to X$ is \emph{conformally parametrised} if the following third order equation is satisfied
	\begin{equation}
		h\left(\dg , |\dg|^{-1}\nabla_{\dg}\left(|\dg|v\right) + \frac{1}{2}|v|^2 \dg - h^{-1}\left( P_h(\dg) \right)  \right)
		=0
		\label{conformal-parametrisation}
	\end{equation}
	Here $\nabla$ denotes the Levi-Civita connection of $h$ and we use the shorthand $v = \nabla_{\dg} \left( |\dg|^{-2}\dg\right)$. Meanwhile,
	\[
	P_h = \frac{1}{n-2} \Ric^0_h + \frac{R_h}{2n(n-1)}h
	\]
	is the Schouten tensor of $h$ (with $R_h$ the scalar curvature and $\Ric_h^0$ the trace-free Ricci curvature). 
	
	One can check that a conformal parametrisation always exists, is uniquely determined up to the action of $\PSL(2,\R)$ and, moreover, the Definition~\ref{conformal-parametrisation} is independent of the choice of $h\in [h]$ (remembering of course that $\nabla$, $v$ and $P_h$ will all change). As presented here, the condition~\eqref{conformal-parametrisation} appears completely mysterious. There is a neat interpretation in terms of tractor calculus \cite{bailey_thomass_1994}. Our next result gives an alternative geometric interpretation of these preferred conformal parametrisations in terms of our ambient surface. (We make connection the with the tractor point of view later.) Given a parametrised curve $u \colon I \to X$ with image $\gamma$, we say that $U \colon I \times [0,1) \to Y$ extends $u$ if $U(s,0) = u(s)$ for all $s\in I$. In the following result, we  treat $I \times [0,1) \subset \R \times [0,\infty)$ as a subset of the upper half-space with hyperbolic metric $g_{\mathrm{hyp}} = t^{-2}(dt^2+ ds^2)$. 
	
	\begin{Theorem}\label{isometric-parametrisation}~
		\begin{itemize}
			\item
			Given a parametrised curve $u \colon I \to X$ with image $\gamma$ there exists an extension $U \colon I \times [0,1) \to Y$ parametrising $\Sigma_\gamma$ and such that $U^*(g) =g_{\mathrm{hyp}} + \Oc(t^2)$, where $\Oc(t^2)$ refers to the norm of a tensor with respect to $g_{\mathrm{hyp}}$.
			\item
			There exists an extension $U$ of $u$ with $U^*(g)=g_{\mathrm{hyp}} +\Oc(t^3)$ if and only if $u$ is a conformal parametrisation of $\gamma$. 
		\end{itemize}
	\end{Theorem}
	
	Next we turn to conformal geodesics. These are a distinguished class of curves in $(X,[h])$, singled out purely by the conformal structure. Fix a representative Riemannian metric $h \in [h]$; a curve $\gamma \colon I \to X$ is a conformal geodesic if it solves the third order equation:
	\begin{equation}\label{CGE-third-order-form}
		\nabla_{\dg} \ddot{\gamma} = 3 \frac{h\left(\dg, \ddot{\gamma}\right)}{|\dg|^2}\ddot{\gamma} -\frac{3}{2} \frac{|\ddot{\gamma}|^2}{|\dg|^2}\dg + |\dg|^2 h^{-1}\left(P_h(\dg)\right) -2 P_h\left(\dg,\dg \right) \dg.
	\end{equation}
	One can check that this definition is conformally invariant: the transformation laws for $\nabla$ and $P_h$ under change of representative metric $h$ ensure that if $\gamma$ solves~\eqref{CGE-third-order-form} for $h$ it also solves it for any metric conformal to $h$. 
	
	Again, in this form~\eqref{CGE-third-order-form} appears mysterious and again there is a concise interpretation of the equation using tractor calculus (given in \cite{bailey_thomass_1994} and also outlined briefly in the main body of the article). Our next result is an alternative geometric interpretation of conformal geodesics from the point of view of the ambient surface. In the following, $\gamma \colon I \to X$ is a conformally parametrised curve and we choose $U \colon I \times [0,1) \to Y$ to be an extension parametrising the ambient surface $\Sigma_\gamma$ which is isometric to $\Oc(t^3)$, as in Theorem~\ref{isometric-parametrisation}. Write $K$ for the second fundamental form of $\Sigma_\gamma$. We treat $|K|$ as a function on $I \times [0,1)$ via the parametrisation $U$.

	\begin{Theorem}
		The ambient surface is asymptotically totally geodesic: $|K| = \Oc(t^2)$. Moreover, $|K| = \Oc(t^3)$ if and only $\gamma$ is a conformal geodesic.
	\end{Theorem}
	
	So a curve is a conformal geodesic precisely when its ambient surface is totally geodesic to higher order than one would normally expect to see.

	\subsection{Global questions}\label{global}
	
	As we have stressed, the discussion of the ambient surface in this article is purely local. We close this introduction with a few comments on the corresponding \emph{global} existence problem. There is great interest, both mathematically and physically, in studying genuine (as opposed to purely formal) Poincar\'e--Einstein metrics, i.e.~Einstein metrics on a compact manifold $Y$ with boundary $X$ which have the form~\eqref{conformally-compact} near $X$. A central question is the Dirichlet problem: given a conformal manifold $(X,[h])$ does it arise as the infinity of a Poincar\'e--Einstein metric on a compact manifold $Y$ with boundary $X$? In general this remains open, but there are a wealth of examples (for example \cite{graham_einstein_1991,calderbank_einstein_2004,mazzeo_gluing_2006}). One can also pose the Dirichlet problem for minimal surfaces: let $Y$ be a compact manifold with boundary $X$ and let $g$ be a Poincar\'e--Einstein metric on the interior; given a closed curve $\gamma$ in $X$, when is it possible to fill $\gamma$ by a closed minimal surface $\Sigma$? This question was answered positively by Anderson in the case of hyperbolic space itself \cite{anderson_complete_1982} and Alexakis and Mazzeo showed that for hyperbolic 3-manifolds it is possible to use degree theory to count the number of minimal solutions \cite{alexakis_renormalized_2010}. The general problem however remains open.

	A natural question in the context of this article is the following refinement of this Dirichlet problem: when is it possible to fill $\gamma$ by a closed surface $\Sigma$ which is a critical point of \emph{renormalised} area? There is, to the best of our knowledge, only one result in this direction, due to Alexakis and Mazzeo \cite{alexakis_renormalized_2010}. They prove that in hyperbolic 3-space $\mathbb{H}^3$ the only surfaces which are critical points of renormalised area are the totally geodesic copies of $\mathbb{H}^2$, which fill the closed conformal geodesics on the boundary $S^3$ (with its standard conformal structure).
	
	Clearly, curves which can be globally filled in this way are special. We can also see this from our above discussion about Dirichlet and Neumann data. Recall that in the local existence problem for minimal surfaces, we had the right to prescribe both Dirichlet data (the curve $\gamma$) \emph{and} the Neumann data, which was set to zero by the requirement that the surface be critical for renormalised area. In the global problem however, one will determine the other, just as one can prescribe either the Dirichlet or Neumann data of a harmonic function on the disk, but not both. From this perspective, critical surfaces are the zeros of a Dirichlet-to-Neumann map. Based on this, one might optimistically hope that generically critical surfaces are isolated (the above case of $\mathbb{H}^3$ is exceptional due to the abundance of symmetries) and even, in some circumstances, for there to be finitely many of them. In any case, one might reasonably expect the critical surfaces and their corresponding boundary curves to form an interesting collection of objects, pertinent to the study of both the Riemannian geometry of $Y$ and the conformal geometry of $X$. 
	
	\subsection{Acknowledgements}
	
	We would like to thank Robin Graham for helpful discussions about conformal geodesics. This work was supported by the Fonds Wetenschappelijk Onderzoek---Vlaanderen (FWO) and the Fonds de la Recherche Scientifique---FNRS under EOS Project Number 30950721 ``Symplectic Techniques''. JF was also supported by ERC consolidator grant ``SymplecticEinstein'' 646649. YH was also supported by an FNRS  \emph{charg\'e de recherche} fellowship.
	
	\section{Conformal geometry and conformal geodesics}\label{Sec: Conformal geometry and conformal geodesics}

	In this section we review well known facts about conformal geometry and conformal geodesics. This will serve mainly to fix the conventions that we will use in the following sections.
	
	\subsection{Conformal manifolds}	
	Let $X$ be an $n$-dimensional manifold. Recall that the bundle of $1$-densities $|\bigwedge|X$ is the real line bundle associated to the frame bundle of $X$ with respect to the representation $M \mapsto |det(M)|^{-1}$ of $GL(n)$. This bundle is always trivial. If $X$ is orientable 1-densities coincide with $n$-forms but on non-orientable manifolds $1$-densities (not $n$-forms) are the right type of objects needed for integration. The \emph{bundle of scales} is $L = \left(|\bigwedge|X\right)^{-\frac{1}{n}}$. This is an oriented real line bundle over $X$ (we will only consider positive sections).
	
	We will say that $(X, \hConf)$ is a \emph{conformal manifold} if $X$ is equipped with a non-degenerate symmetric bilinear form $\hConf$ with values in $L^2$, i.e $\hConf$ is a section of $L^2 \otimes S^2 T^*X$ ($S^2$ stands for symmetric tensor product). A choice of scale $\tau \in \Gamma\left[L\right]$ then amounts to a choice of ``representative'' $h = \tau^{-2}\hConf \in \Gamma\left[S^2 T^*X \right]$.
	
\subsection{Conformal geodesics}\label{subsec: Conformal geodesics}

Let $(X , \hConf)$ be an $n$-dimensional conformal manifold. We  briefly review some facts about conformal geodesics when $n>2$, the particular case $n=2$ is also discussed at the end of this subsection. Our main references are \cite{bailey_conformal_1990}, \cite{bailey_thomass_1994} and \cite{tod_examples_2012}, see also \cite{eastwood_uniqueness_2014}. 

Let $\gamma \from \R \to X$ be a curve in $(X, \hConf)$ parametrised by $s$. It is a \emph{parametrised conformal geodesic} if and only if it satisfies the following third order differential equation :
\begin{equation}\label{1- Conformal Geodesic Equation}
F\left(\gamma , h\right) :=\nabla_{\dg} \ddot{\gamma} - 3 \frac{h\left(\dg, \ddot{\gamma}\right)}{|\dg|^2}\ddot{\gamma} +\frac{3}{2} \frac{|\ddot{\gamma}|^2}{|\dg|^2}\dg - |\dg|^2 h^{-1}\left(P_h(\dg)\right) +2 P_h\left(\dg,\dg \right) \dg =0.
\end{equation}
 Here $h$ is a choice of representative for $\hConf$, $\nabla$ the associated Levi-Civita connection, $\dg := \partial_s \gamma$, $\ddot{\gamma} := \nabla_{\dg}\dg$, and $P_h = \frac{1}{n-2}Ric_h\big|_{0} + h\; \frac{1}{2n(n-1)} R_h$ is the Schouten tensor of $h$ (here and everywhere in this article $\big|_{0}$ denotes the ``trace-free part'' of a tensor) . One can show that for any non-vanishing function $\gO \in \mathcal{C}^{\infty}\left(X\right)$ one has $F\left(\gamma , \gO^{2}h\right) = F\left(\gamma , h\right)$ (this follows from the conformal transformation rules for the Levi-Civita connection and the Schouten tensor), see \cite{bailey_conformal_1990} for a proof. The vanishing of \eqref{1- Conformal Geodesic Equation} therefore does not depend on the choice of representative for $[h]$.

 It is also shown in \cite{bailey_conformal_1990} that the normal part of \eqref{1- Conformal Geodesic Equation} (i.e.\ the component which is orthogonal to~$\dg$) does not depend on the chosen parametrisation.  We thus call (unparametrised) \emph{conformal geodesics} those curves which satisfy the normal projection of \eqref{1- Conformal Geodesic Equation}. They can be thought as analogues of geodesics from Riemannian geometry with the normal part of \eqref{1- Conformal Geodesic Equation} corresponding to the geodesic equations.

In Riemannian geometry there is a natural arc-length parametrisation for curves (unique up to a remaining linear reparametrisation freedom). The tangential part of \eqref{1- Conformal Geodesic Equation}, $h\left(\dg , F\left(\gamma, h\right)\right) =0$, gives the analogue for conformal geometry: as is shown in \cite{bailey_conformal_1990} one can always (at least locally) choose a parametrization such that this equation is satisfied. The remaining freedom in reparametrisation is then the projective linear group
\begin{equation}
s \mapsto \frac{as+b}{cs+d}, \qquad\text{with} \begin{pmatrix}
a & b\\c & d
\end{pmatrix} \in PSL(2, \R).
\end{equation}
We will say that curves satisfying the tangential part of \eqref{1- Conformal Geodesic Equation} have been given a \emph{preferred conformal parametrisation}.

For most purposes, Equation \eqref{1- Conformal Geodesic Equation} is cumbersome. Following \cite{tod_examples_2012}, we introduce an auxiliary vector field $v$ and \eqref{1- Conformal Geodesic Equation} is found to be equivalent to
 \begin{subequations}
 	\begin{align}[left = \empheqlbrace\,]
 	 v &=\nabla_{\dg}\left(|\dg|^{-2} \dg \right)  \label{1- Eq: First Conformal Geodesic Eq}\\
 	0 &=|\dg|^{-1}\nabla_{\dg}\left(|\dg|v\right) + \frac{1}{2}|v|^2 \dg - h^{-1}\left( P_h(\dg) \right) . \label{1- Eq: Second Conformal Geodesic Eq}
 	\end{align}
 \end{subequations} 
This is this form that we will mainly use in the following. It will be especially well adapted when we come to tractors but it will also appear naturally in the holographic discussion.

We now briefly discuss the $n=2$ case. Since the Schouten tensor is only well defined for $n>2$ the same is true for the conformal geodesic equations \eqref{1- Conformal Geodesic Equation}. One can however extend \eqref{1- Conformal Geodesic Equation} to the two dimensional case by making a choice of M\"obius structure (see \cite{calderbank_mobius_2006}). A pragmatic point of view on M\"obius structures is that they amounts to a choice of traceless symmetric tensor $P_0$ whose behaviour under change of representatives of the conformal class mimics that of the traceless Ricci tensor $Ric_h|_0$. In the following we assume that we have made a such a choice and all formula then straightforwardly extend to the case $n=2$ by replacing the (missing) trace-free Schouten tensor by this $P_0$. See however \cite{calderbank_mobius_2006}, \cite{burstall_conformal_2010} for more details on the underlying geometry.

	\section{The ambient surface and conformal geodesics}\label{Sec: Holographic Prescription for Conformal Geodesics}
	
	\subsection{The Poincar\'e--Einstein metric of a conformal manifold}

Let $(X , \hConf)$ be an $n$-dimensional conformal manifold. Let $(Y , g)$ be the associated Poincar\'e--Einstein (formal) manifold defined in terms of the Fefferman--Graham expansion \cite{fefferman_conformal_1985}, \cite{fefferman_ambient_2012}. By construction, $Y$ is the interior of a $n+1$-dimensional compact manifold $\bar{Y}$ with boundary $\partial \bar{Y} = X$ and if $x \from \bar{Y} \to \R$ is boundary defining function, $X= x^{-1}(0)$, $dx|_X \neq 0$ then $\bar{g} = x^2 g$ smoothly extends to $\bar{Y}$. The leading order of the expansion for $g$ is
\begin{equation}\label{3- FG Third Normal Form2}
g = \frac{1}{x^2} \left(dx^2 + \left[ h - x^2 P + \Oc\left(x^4\right) \right] \right).
\end{equation}
Where $P$ is the Schouten tensor of $h$ if $n>2$, and is given by a choice of M\"obius structure if $n=2$. 
The boundary defining function $x$ used to write down this expansion is not unique. However once we make a choice of representative $h$ one can always find $x$ such that the expansion is of the form \eqref{3- FG Third Normal Form2} and this fixes $x$ uniquely, see \cite{fefferman_conformal_1985}, \cite{fefferman_ambient_2012}. Such boundary defining functions are called ``special'' or ``geodesic''.

In what follows we won't need the precise form of the expansion beyond third order, so that for the rest of this work Equation \eqref{3- FG Third Normal Form2} (formally) defines the metric $g$.

\subsection{The ambient surface of a conformal curve}\label{subsec: Dual surface to a conformal curve}

We now take $\g \from [0,1] \to X$ to be a curve in $X$ parametrised by a parameter $s$.
\begin{Definition}\label{3- Def: Surface with asymptotic boundary}
	 We will say that an embedding $\gS \from [0,1] \times [0,1) \to \bar{Y}$ is a \emph{surface with asymptotic boundary $\g$} if 
	 \begin{itemize}
	 	\item $\gS\left(s, 0\right)= \g(s)$,
	 	\item the image of $\gS$ intersects the asymptotic boundary at right angles, $\gS \perp X$.
	 \end{itemize}
Orthogonality in the second point is taken with respect to the metric $\bar{g}$. This metric is only defined up to scale, but orthogonality does not depend on the particular choice of scale. In what follows we will frequently abuse notation by using $\gS$ to denote the restriction to $[0,1]\times(0,1)$, whose image lies in the interior $Y$.
\end{Definition}

Let $x$ be a choice of boundary defining function. If $\gS$ is a surface with asymptotic boundary $\g$, we take $\gS_{\epsilon}$ to be defined by $\gS_{\epsilon} \coloneqq \gS \cap \left\{ x \geq \epsilon \right\}$ and its boundary to be $\gamma_{\epsilon} = \partial \gS_{\epsilon} $. As is explained in \cite{graham_conformal_1999}, the area of $\gS_{\epsilon}$ diverges as the length of $\g_{\epsilon}$, one can however ``cut off'' this diverging part to obtain the renormalised area.
\begin{Lemma}{(from \cite{graham_conformal_1999})}
Let $\mathcal{A}\left(\gS \right)$ be a surface $\gS$ with asymptotic boundary $\g$ and let $dA$, $dl$ be the volume forms induced by the Poincar\'e--Einstein metric $g$ on $\gS$ and $\g_{\epsilon}$. The limit
\begin{equation}\label{3- Renormalised Area}
\mathcal{A}\left(\gS \right) = lim_{\epsilon \to 0} \left( \int_{\gS_{\epsilon}} dA - \int_{\g_{\epsilon}} dl \right).
\end{equation}
is finite and does not depends on the choice of boundary defining function. The resulting limit $\mathcal{A}\left(\gS \right)$ is called the renormalized area of $\gS$.
\end{Lemma}

The next theorem is the main result of this subsection. It allows us to associate to a curve $\g$ in $X$ a unique ambient surface $\gS_{\g}$ in $Y$.
\begin{Theorem}\label{3- Thrm: Holographic dual}\mbox{}\\
Let $\g$ be a curve in $\left(X ,[h]\right)$, let $\left(Y , g \right)$ be the (formal) associated Poincaré--Einstein manifold \eqref{3- FG Third Normal Form2}, then there is a unique (formal) surface $\gS_{\g}$ in $Y$ which both is a critical point of the renormalised area \eqref{3- Renormalised Area} and has asymptotic boundary $\g$.

What is more, this surface can be asymptotically described by the expansion,
\vspace*{-0.6cm}
\begin{adjustwidth}{-9pt}{-0pt} 
\begin{equation}\label{3- Holographic dual, expansion}
\gS_{\g} = \left\{\begin{array}{ccccccccccc}
x\left(s,t\right) & =& 0 &+& t |\dg| &+& 0 &+&\frac{t^3 }{3}\left(-\frac{3}{4}|\dg|^3|v|^2 + \frac{1}{2} |\dg| \kappa(\g,v, h) \right) &+& \Oc\left(t^4\right) \\ \\
y^i\left(s,t\right) & =& \g^i(s) &+& 0  &+& \frac{t^2}{2} |\dg|^2 v^i &+& 0 &+& \Oc\left(t^4\right)
\end{array} \right\}
\end{equation}
\end{adjustwidth}
where $v$ satisfies the first conformal geodesic equation \eqref{1- Eq: First Conformal Geodesic Eq}:
\begin{equation}\label{3- leading minimal surface eq }
v =\nabla_{\dg}\left(|\dg|^{-2} \dg \right)
\end{equation}
and
\begin{equation}\label{3- Conformal Parametrisation Def1}
\kappa(\g,v, h) \coloneqq h\left(\dg , |\dg|^{-1}\nabla_{\dg}\left(|\dg|v\right) + \frac{1}{2}|v|^2 \dg - h^{-1}\left( P_h(\dg) \right)  \right)
\end{equation}
where the right-hand-side of \eqref{3- Conformal Parametrisation Def1} is the tangential part of the right-hand-side of the second conformal geodesic equation \eqref{1- Eq: Second Conformal Geodesic Eq}.

The expansion \eqref{3- Holographic dual, expansion} is not unique but can be invariantly defined as a choice of conformal parametrisation up to order three, that is such that the induced metric is, in the coordinate basis $\{\partial_s, \partial_t\}$,
\begin{equation}\label{3- Induced metric1}
\gS^*g = f\left(s,t\right)\left[\begin{pmatrix}
1 & 0 \\0 & 1
\end{pmatrix} + \Oc(t^3)\right].
\end{equation}
\end{Theorem}
 We will say that the surface $\gS_{\g}$ given by the above theorem is the ambient surface associated to the curve $\g$. In the following we will not use the details of $\gS_{\gamma}$ beyond third order in the parameter~$t$, consequently we could just as well have taken the expansion \eqref{3- Holographic dual, expansion} as our definition for the ambient surface.

We break down the proof of the above theorem into three lemmas from which it is a direct consequence.
\begin{Lemma}\label{3- Lemma: Conformal parametrisation}
Let $\left(X , \hConf \right)$ be a conformal manifold, let $\left(Y , g \right)$ be the associated Poincaré--Einstein manifold and let $x$ be a boundary defining function such that $g$ satisfies \eqref{3- FG Third Normal Form2}. Let $\g$ be a curve in $X$ and let $\gS$ be a surface in $Y$ with asymptotic boundary $\g$, then $\gS$ can always be described by the expansion
\vspace*{-0.4cm}
\begin{adjustwidth}{-5pt}{-0pt} 
\begin{equation}\label{3- Asymptotic surface, expansion}
\gS = \left\{\begin{array}{ccccccccccc}
x\left(s,t\right) & =& 0 &+& t |\dg| &+& 0 &+&\frac{t^3 }{3}\left(-\frac{3}{4}|\dg|^3|v|^2 + \frac{1}{2} |\dg| \kappa(\g, v, h) \right) &+& \Oc\left(t^4\right) \\ \\
y^i\left(s,t\right) & =& \g^i(s) &+& 0  &+& \frac{t^2}{2} |\dg|^2 v^i &+& \frac{t^3}{3}\; n^i &+& \Oc\left(t^4\right)
\end{array} \right\}
\end{equation}
\end{adjustwidth}
where $v$ and $n$ respectively satisfy $h\left(\dg , v \right) =  h\left(\dg , \nabla_{\dg}\left(|\dg|^{-2} \dg \right) \right)$, $h\left(\dg, n\right)=0$ and $\kappa(\g,v, h)$ is given by \eqref{3- Conformal Parametrisation Def1} (note that $v$ is however not supposed to satisfy \eqref{3- leading minimal surface eq }).

This description is not unique but can be invariantly defined as a choice of conformal parametrisation up to order three,
\begin{equation}\label{3- Induced metric2}
\gS^*g = f(s,t) \left[ \begin{pmatrix}
1 & 0 \\0 & 1
\end{pmatrix}  + \Oc(t^3) \right].
\end{equation} 
\end{Lemma}

\begin{proof}\mbox{}\\
Let $\gS$ be a surface in $Y$ with asymptotic boundary $\g$, by Definition~\ref{3- Def: Surface with asymptotic boundary} it must have an expansion of the form:
\begin{equation}\label{3- Lemma(conformal coordinate)2}
\gS = \left\{\begin{array}{ccccccccccc}
x\left(s,t\right) & =& 0 &+& t x_1 &+& \frac{t^2}{2} x_2 &+& \frac{t^3}{3} x_3 &+& \Oc\left(t^4\right) \\ \\
y^i\left(s,t\right) & =& \g^i(s) &+& t\; \alpha(s)\; \dg^i(s)  &+& \frac{t^2}{2} |\dg|^2 v^i &+& \frac{t^3}{3}\;n^i &+& \Oc\left(t^4\right)
\end{array} \right\}.
\end{equation}
We will prove the lemma by explaining how to fix the coefficients in \eqref{3- Lemma(conformal coordinate)2} to agree with \eqref{3- Asymptotic surface, expansion}.
This is done by requiring that the coordinates used in the expansion are ``asymptotically isothermal'' i.e satisfy \eqref{3- Induced metric2}. Solving order by order Equation \eqref{3- Induced metric2} results in the expansion \eqref{3- Asymptotic surface, expansion} with $h\left(\dg , v \right) =  h\left(\dg , \nabla_{\dg}\left(|\dg|^{-2} \dg \right) \right)$, $h\left(\dg, n\right)=0$ and $ \kappa(\g, v, h)$ given by \eqref{3- Conformal Parametrisation Def1}.
\end{proof}

From the detailed derivation of the above proof one deduces the exact form of the function $f$ in \eqref{3- Induced metric2}.
\begin{Corollary}\label{3- Crllr: Induced metric}
Let $\gS$ be a surface with asymptotic boundary $\g$ in the parametrisation \eqref{3- Asymptotic surface, expansion} given by lemma \ref{3- Lemma: Conformal parametrisation} then the induced metric is, in the coordinate basis $\{\partial_s, \partial_t\}$,
\begin{equation}\label{3- Induced metric}
\gS^*g = \frac{1}{t^2}
\begin{pmatrix}
1 & 0\\ 0 & 1
\end{pmatrix}\left(1 + 0.t +  t^2\left( \frac{2}{3} \kappa(\g, v, h) \right)\right) + \Oc(t).
\end{equation}
\end{Corollary}

The next two lemmas are generalisations to arbitrary dimension of results established in \cite{alexakis_renormalized_2010} for $n=2$.

We first consider variations of the renormalised area \eqref{3- Renormalised Area} with compact support. Let $\g$ be a curve in $\left(X ,[h]\right)$, let $\left(Y , g \right)$ be a Poincar\'e--Einstein manifold with $g$ of the form~\eqref{3- FG Third Normal Form2}. Let $\gS(u)$ be any one dimensional family of surfaces in $Y$ with asymptotic boundary $\g$, in particular it must have fixed boundary $\partial \gS(u) = \g$. Let $\gS(u)$ start at a surface $\gS(0) = \gS$ parametrised as in~\eqref{3- Asymptotic surface, expansion}. We write $\gS(u)$ as 
\begin{equation}\label{3- Lemma(first variation)1}
\gS(u) = \gS + u \Phi + \Oc(u^2) 
\end{equation} where $\Phi \in \gG[N_{\gS}]$ is a section of the normal bundle to $\gS$ in $Y$ (more correctly $\gS(u)$ is a path of surfaces with first variation $\Phi$). Let us remark that, since $\gS(u)$ is supposed to have asymptotic boundary $\g$, we must have $\Phi\big|_X=0$. In fact, since the surfaces $\Sigma(u)$ meet $X$ at right angles, $\Phi$ vanishes to first order along $X$.

Before we get to our next lemma, let us introduce some notation. If $\gS \subset Y$ is a surface in $Y$, we will note $K$ the induced second fundamental form and $Tr K$ the mean curvature (defined as the trace of $K$).

\begin{Lemma}\label{3- Lemma: first variations with compact support}
The renormalised area of the surface $\Sigma$ is stationary for  variations fixing the boundary if and only if the surface $\gS$ is minimal, that is if the mean curvature vanishes. What is more, setting $Tr K$ to zero fixes $\gS$ (at least its formal expansion) uniquely up to the choice of the normal part of the coefficient $n$ in the expansion \eqref{3- Asymptotic  surface, expansion}. In particular the mean curvature vanishes to leading order in $t$ if and only if $v$ in \eqref{3- Asymptotic surface, expansion} satisfies the first conformal geodesic equation \eqref{1- Eq: First Conformal Geodesic Eq}:
 \begin{equation}\label{3- Minimality, leading order}
v =\nabla_{\dg}\left(|\dg|^{-2} \dg \right).
\end{equation}
\end{Lemma}

\begin{proof}(First part of)\mbox{}\\
Since we only consider variations fixing the asymptotic boundary $\partial \gS(u) = \g$, the variation of the renormalised area \eqref{3- Renormalised Area} coincides with the variation of the area. The surface $\gS$ is thus critical for such variations if and only if its mean curvature vanishes. As was already discussed in the introduction it is known that this equation has indicial roots $0$ and $3$. (See, for example, \cite{alexakis_renormalized_2010} or \cite{graham_higher-dimensional_2017}.) This implies that one can formally solve for this equation order by order in the expansion \eqref{3- Asymptotic surface, expansion} and that the solution is unique up to a choice of terms at order $0$ and $3$. A direct computation with the parametrization \eqref{3- Asymptotic surface, expansion} then shows that these ``free'' terms are $\g$ and the normal part of $n$ (recall that, from our choice of coordinates, $h\left(\dg, n\right)=0$ ). All other terms are then uniquely constrained by solving $Tr K$ order by order. In particular one finds that $v$ must satisfy the first conformal geodesic equation \eqref{1- Eq: First Conformal Geodesic Eq}, we however postpone a precise derivation of this fact to the next section (see Corollary \ref{3- Crllr: mean curvature}).
\end{proof}

We now consider the variations that do not necessarily fix the conformal boundary. Let $\gS$ be a minimal surface, $Tr K =0$, with asymptotic boundary $\g$. Let $\g(u)$ be any one dimensional family of curves in $X$ starting at $\g(0)=\g$ and let $\gS(u)$ be any one dimensional family of surfaces in $Y$ with asymptotic boundary $\g(u)$ and starting at $\gS(0)= \gS$. 
We write $\g(u)$ as
\begin{equation}\label{3- Lemma(first variation)2}
\g(u) = \g + u \phi + \Oc(u^2) 
\end{equation}
where $\phi \in \gG[N_{\g}]$ is a section of the normal bundle to $\g$ in $X$ (as for equation \eqref{3- Lemma(first variation)1}, this means that $\gamma(u)$ is a path of curves with first variation $\phi$). 

\begin{Lemma}\label{3- Lemma: first variation with non-compact support}
The first variation of the renormalised area with respect to this family of surface is
\begin{equation}
\left.\frac{d \mathcal{A}\left( \gS_t \right) }{d t}\right|_{t=0} = -\int_{\g}ds\;|\dg|^{-2} \; h\left( \phi, n\right) .
\end{equation}
This vanishes if and only if the normal part of $n$ in \eqref{3- Asymptotic surface, expansion} vanishes and thus if and only if
\begin{equation}
n=0.
\end{equation}
\end{Lemma}

\begin{proof}\mbox{}\\
	 Let $\gS_{\epsilon}(u) \coloneqq \gS(u) \cap \left\{ x \geq \epsilon \right\}$ and $\gamma_{\epsilon}(u) = \partial \gS_{\epsilon}(u) $ its boundary. Let $\Phi$ be a section of the normal bundle defined by \eqref{3- Lemma(first variation)1}.
 We denote $dA(u)$ the volume-form induced by $g$ on $\gS(u)$ and, finally, we denote by $dl_{\epsilon}$ the volume-form induced by $g$ on $\g_{\epsilon}$.

We first remark that
\begin{align}
\left.\frac{d }{d u} \left(\int_{\gS_{\epsilon}(u)} dA(u)\right)\right|_{u=0} &= \int_{\gS_{\epsilon}} \left.\frac{d}{d u} \left(dA(u)\right)\right|_{u=0} + \left.\frac{d }{d u}\left(\int_{\gS_{\epsilon}(u)} dA \right)\right|_{u=0},
\end{align}
and treat each of these integrals separately. The first term is
\begin{align}\label{3- Lemma(first variation)5}
\begin{array}{lll}
\int_{\gS_{\epsilon}} \left.\frac{d}{d u} \left(dA(u)\right)\right|_{u=0}&= - \int_{\gS_{\epsilon}} tr_{g} K(\Phi) dA.
\end{array}
\end{align}
The second integral can be rewritten as
\begin{align}\label{3- Lemma(first variation)6}
\begin{array}{lll}
\left.\frac{d }{d u}\left(\int_{\gS_{\epsilon}(u)} dA \right)\right|_{u=0} &=\left.\frac{d }{d u}\left( \int_{\epsilon = \gS_u^* x }^{\infty}\int_{\g_{\epsilon}} dA \right)\right|_{u=0}
&=\left.\frac{d }{d u}\left(\int_{\epsilon= \gS^* x + dx(\Phi) u  }^{\infty}\int_{\g_{\epsilon}} dA\right)\right|_{u=0}\\ \\
&=  \int_{\g_{\epsilon}} dl \;\frac{ dx(\Phi)}{|\gS^* dx|^2}\left| \nabla x \right|
&=  \int_{\g_{\epsilon}} dl \;\frac{ dx(\Phi)}{|\gS^* dx|}.
\end{array}
\end{align}
The second line follows from the fact that, in order to increase the function $\gS^*x$ by $dx(\Phi)$, one needs to follow the gradient flow $\nabla x = g^{-1} (\gS^* dx)$ for a time $\frac{dx(\Phi)}{|\gS^* dx|^2}$: \begin{equation}
\gS^*dx\left(\frac{dx(\Phi)}{|\gS^* dx|^2} \nabla x\right) = dx(\Phi).
\end{equation} Combining \eqref{3- Lemma(first variation)5} and \eqref{3- Lemma(first variation)6}, we have
\begin{align}\label{3- Lemma(first variation)7}
\left.\frac{d }{d u} \left(\int_{\gS_{\epsilon}(u)} dA(u)\right)\right|_{u=0} &= \int_{\g_{\epsilon}} dl\;  \frac{dx\left(\Phi \right)}{|\gS^* dx|} - \int_{\gS_{\epsilon}} tr_{g} K(\Phi) dA.
\end{align}

We now expand the above in powers of epsilon. One will need an asymptotic expansion of a generic element $\Phi$ of the normal bundle. This needs a bit of work and will be given a proof of its own (see Lemma~\ref{3- Lemma: Normal vector fields} and the following proof). We here only state the end result: Let $\Phi$ be a section of the normal bundle of $\gS$ in $Y$, then there exist a one parameter family $\phi(t)$ of sections of the normal bundle to $\g$ in $X$ such that
\begin{equation}
\Phi\left(s,t\right) = \phi(t) - h\left(\phi(t), t|\dg|v + t^2|\dg|^{-1}n\right) \partial_x -\frac{t^2}{2} \alpha \dg + \Oc(t^3).
\end{equation}
Where $\alpha$ is a function of $\phi(t)$ and $v$. In particular
\begin{equation}
dx\left(\Phi\right) = - h\left(\phi(t), t|\dg|v + t^2|\dg|^{-1}n\right)  +\Oc(t^3).
\end{equation}
Putting this with \eqref{3- Lemma(first variation)7} and the expansion of $\gS^*x$ given by \eqref{3- Asymptotic surface, expansion}, we get:
\begin{align}\label{3- Lemma(first variation)3}
\left.\frac{d }{d u} \Bigg(\int_{\gS(u)} dA(u)\Bigg)\right|_{u=0}
=-\frac{1}{\epsilon}\; \int_{\g} ds\; h\left(\phi , |\dg| v \right)- \int_{\g} ds |\dg|^{-2} h\left(\phi, n\right)  - \int_{\gS_{\epsilon}} tr_{g} K(\Phi) dA + \Oc(\epsilon).
\end{align}
Finally, noting that
\begin{equation}
\left.\frac{d}{d u}\left( \int_{\g_{\epsilon}(u)} dl(u)\right)\right|_{u=0} = -\frac{1}{\epsilon}\int_{\g} ds\;|\dg| \;  h\left(\phi , \nabla_{\dg} \left(|\dg|^{-2} \dg \right) \right),
\end{equation}
we can recast Equation \eqref{3- Lemma(first variation)3} as
\begin{align}\label{3- Lemma(first variation)4}
&\left.\frac{d }{d u} \left(\int_{\gS_{\epsilon}(u)} dA(u) - \int_{\g_{\epsilon}(u)} dl(u) \right) \right|_{u=0} \phantom{\hspace{8cm}}
\end{align}
\begin{align*}
=- \int_{\gS_{\epsilon}} tr_{g} K(\Phi) dA - \frac{1}{\epsilon} \; \int_{\g} ds\;|\dg|  \; h\bigg(\phi , \Big(v -\nabla_{\dg} \left(|\dg|^{-2} \dg\right)\Big)\bigg)- \int_{\g} ds |\dg|^{-2} h\left(\phi, n\right) + \Oc(\epsilon) .\nonumber 
\end{align*}
One can show that the divergences in the two first terms in this expansion compensate each others so that the first variation of the renormalised area is well defined. 

For variations fixing the conformal boundary i.e $\phi = 0$, we recover the variation of the area:
\begin{equation}
\frac{d \mathcal{A}\left( \gS_u \right) }{d u}|_{u=0} = \lim_{\epsilon=0}\;\left.\frac{d }{d u} \left(\int_{\gS_{\epsilon}(u)} dA(u) - \int_{\g_{\epsilon}(u)} dl(u) \right) \right|_{u=0} = - \int_{\gS} tr_{g} K(\Phi) dA.
\end{equation}
For $\phi \neq 0$ and if we suppose that $\gS$ is minimal the diverging terms vanish as a consequence of Lemma \ref{3- Lemma: first variations with compact support} Equation \eqref{3- Minimality, leading order} and the variation of the renormalised area is:
\begin{equation}
\left.\frac{d \mathcal{A}\left( \gS_u \right) }{d u}\right|_{u=0} = \lim_{\epsilon=0}\;\left.\frac{d }{d u} \left(\int_{\gS_{\epsilon}(u)} dA(u) - \int_{\g_{\epsilon}(u)} dl(u) \right) \right|_{u=0} = -\int_{\g} ds |\dg|^{-2} h\left(\phi, n\right).
\end{equation}
This concludes the proof of Lemma~\ref{3- Lemma: first variation with non-compact support}.
\end{proof}

\subsection{Conformal geodesics and their ambient surfaces}
We now rephrase the results from section \ref{subsec: Conformal geodesics} on curves $\g \from I \to X$ in a conformal manifold $\left(X, [h]\right)$ in terms of their ambient surface $\gS_{\g} \from I^2 \to Y$ in the associated Poincaré--Einstein manifold $\left(Y, g\right)$ as defined in the preceding subsection.

Let us first start with a remark. If $\g$ is a  curve in $X$ and $\gS_{\g}$ is its ambient surface given by Theorem~\ref{3- Thrm: Holographic dual} then $\gS_{\g}$ is ``asymptotically hyperbolic'' in the sense that its Gauss curvature asymptotically behaves as $G_{\gS}= -1+ \Oc(t^3)$. (This can be directly obtained from the expansion of the induced metric given by Corollary~\ref{3- Crllr: Induced metric}).

The following theorem then answers the natural question ``What does it take to make the induced metric on $\gS_{\g}$ to be as close as possible to the two-dimensional hyperbolic half-plane model?''.

\begin{Theorem}\label{3- Thrm: Conformal Parametrisation}\mbox{}\\
Let $\g$ be a curve in $X$ and let $\gS_{\g}$ be its ambient surface given by Theorem~\ref{3- Thrm: Holographic dual}. 
We suppose that $\gS_{\g}$ is given the parametrisation~\eqref{3- Asymptotic surface, expansion}. The induced metric then asymptotically approximates the Poincar\'e half-plane metric up to second order,
\begin{equation}\label{3- Induced metric3}
\gS^*_{\g}g = \frac{1}{t^2} \left[
\begin{pmatrix}
1 & 0\\ 0 & 1
\end{pmatrix}\left(1 + t^2\left(\frac{2}{3}\kappa(\g, v, h)\right)\right) + \Oc(t^3) \right].
\end{equation}
(Recall the definition \eqref{3- Conformal Parametrisation Def1} of $\kappa(\g, v, h)$.) What is more, the curve $\g$ satisfies the tangential part of the conformal geodesic equations~\eqref{1- Conformal Geodesic Equation} if and only if the induced metric is
\begin{equation}
\gS^*_{\g}g = \frac{1}{t^2} \left[
\begin{pmatrix}
1 & 0\\ 0 & 1
\end{pmatrix} + \Oc(t^3) \right].
\end{equation}
In other words a choice of preferred conformal parametrisation for 
$\g$ amounts to choosing the parametrisation of its ambient surface $\gS_{\g}$ which makes it looks as close as is possible to the hyperbolic half-plane model.
\end{Theorem}	

\begin{proof}\mbox{}\\
This is a nearly direct consequence of the combination of Theorem~\ref{3- Thrm: Holographic dual} with Corollary~\ref{3- Crllr: Induced metric}: the expansion \eqref{3- Induced metric3} is given by Corollary~\ref{3- Crllr: Induced metric} while Theorem~\ref{3- Thrm: Holographic dual}  implies that for the ambient surface $\gS_{\g}$ associated to $\g$ the vector field $v$ in the parametrisation~\eqref{3- Holographic dual, expansion} satisfies the first geodesic equation~\eqref{1- Eq: First Conformal Geodesic Eq}. It follows that $\kappa(\g, v, h)$ as defined by Equation \eqref{3- Conformal Parametrisation Def1} is the tangential part of the conformal geodesic equation~\eqref{1- Conformal Geodesic Equation} whose vanishing defines a preferred conformal parametrisation.
\end{proof}

Theorem \ref{3- Thrm: Conformal Parametrisation} gives a geometric understanding of the vanishing of the \emph{tangential} part of the conformal geodesic equation \eqref{1- Conformal Geodesic Equation} for a curve $\g$ in terms of its ambient surface $\gS_{\g}$, the next theorem does the same for the normal part.

\begin{Theorem}\label{3- Thrm: Conformal Geodesics}\mbox{}\\
Let $\g$ be a curve in $X$ and let $\gS_{\g}$ be its ambient surface given by Theorem \ref{3- Thrm: Holographic dual}. 
The norm  $|K|$ of the extrinsic curvature of the ambient surface vanishes asymptotically as
\begin{equation}
|K| = 0 + \Oc\left(t^2\right).
\end{equation}
What is more, the curve $\g$ satisfies the normal part of the conformal geodesic equation \eqref{1- Conformal Geodesic Equation} if and only if the above norm vanishes one order higher than expected
\begin{equation}
|K| = 0 + \Oc\left(t^3\right).
\end{equation}
\end{Theorem}

The rest of this section is dedicated to proving the above Theorem. Even though straightforward, the proof is a bit lengthy and we will split it in three lemmas and two corollaries. The reader also ought to know that we will present in section \ref{subsubsec: Extrinsic curvature in the ambient metric} a much shorter proof, however relying on the use of tractors and the ambient construction.

\begin{Lemma}\label{3- Lemma: Normal vector fields}
Let $\gS$ be any surface in $Y$ with asymptotic boundary $\g$. Let $\Phi \in \Gamma\left[N_{\gS}\right]$ be a section of its normal bundle in $Y$. Then there exists a one parameter family $\phi(t)$ of sections of the normal bundle to $\g$ in $X$ such that
\begin{align}\label{3- Normal vector fields}
\Phi\left(s,t\right) &= \phi(t) - h\left(\phi(t), t|\dg|v + t^2|\dg|^{-1}n\right) \partial_x \\&\quad -\frac{t^2}{2}  \Big( h\left(\phi(t) , \partial_s v\right) + \left(v^i\partial_i h -2P\right)(\phi(t), \dg) \Big)\dg + \Oc(t^3)\nonumber
\end{align}
\end{Lemma}

\begin{figure}[h]
	\begin{center}
		\includegraphics[scale=0.5]{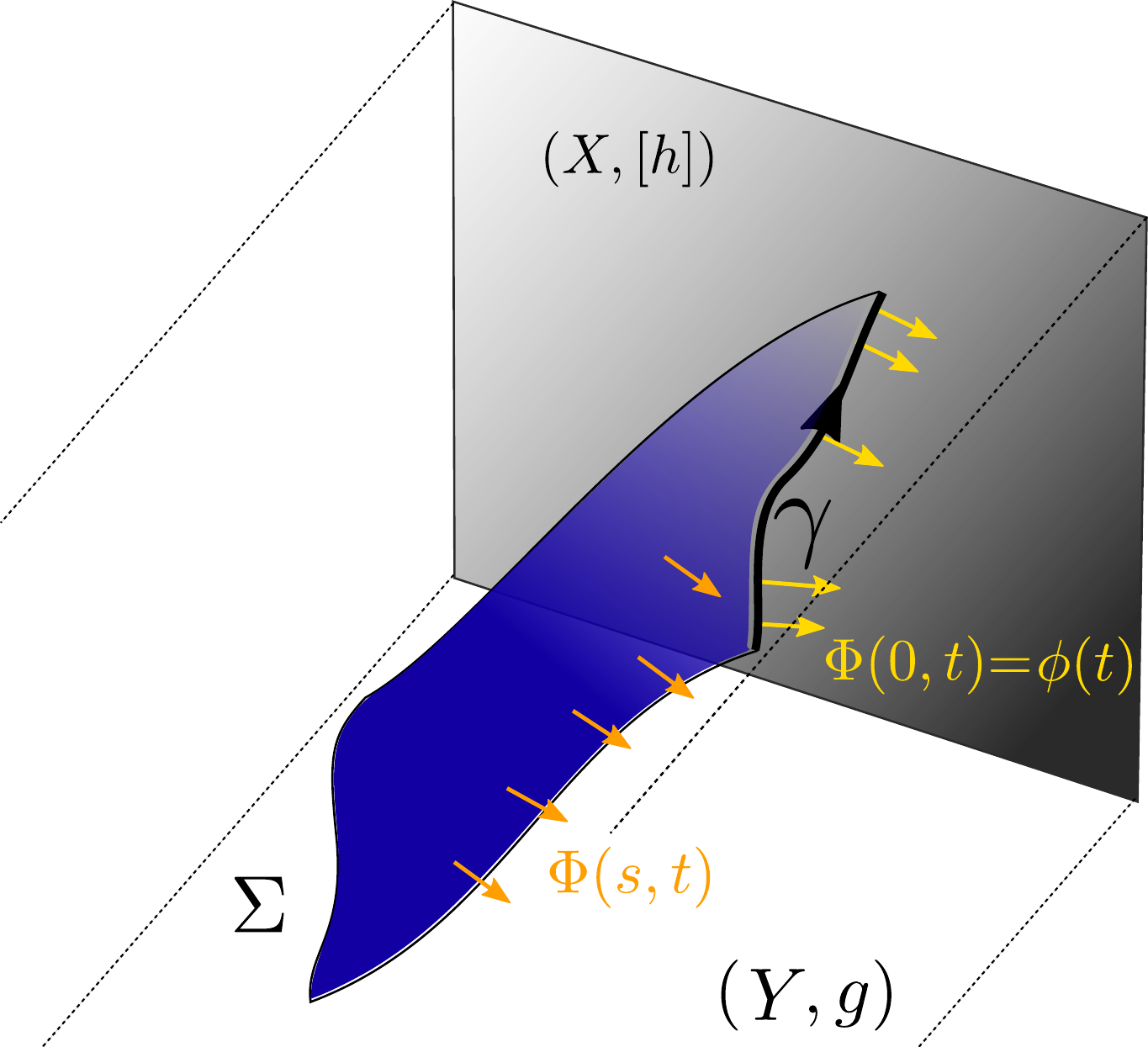}
	\end{center}
	\caption{A normal vector field $\Phi \in \Gamma\left[N_{\gS}\right]$ to a surface $\gS \subset Y$ with asymptotic boundary $\g \subset X$}.
	\label{fig:AmbientSurfaceWithNormals}
\end{figure}

\begin{Lemma}\label{3- Lemma: Normal vector fields bis}
	Let $\gS$ be any surface in $Y$ with asymptotic boundary $\g$ and let $\left\{ \eta_a\right\}_{a \in 1...n-1}$ be a basis of $N_{\g}$ the normal bundle to $\g$ in $X$.	Then the set $\left\{ N_a\right\}_{a \in 1...n-1}$ of vector fields asymptotically obtained by taking $\phi(t)=\eta_a$ in the previous lemma gives, for any point $p\in\gS$, a basis of $N_p\gS$, the normal bundle to $\gS$ at $p$.

What is more, their inner product is
\begin{equation}\label{3- Normal vector field, inner product}
g\left(N_a, N_b\right) = \frac{1}{t^2}\bigg(h\left(\eta_a , \eta_ b\right) |\dg|^{-2} + 0 + \Oc(t^2)\bigg).
\end{equation}
\end{Lemma}
\begin{proof}\mbox{}\\
Let $\Phi \in \gG\left[N_{\gS}\right]$ be any section of the normal tangent bundle:
\begin{equation}
\Phi = \Phi_0 + n_0 \partial_x + t\bigg( \Phi_1 + n_1 \partial_x  \bigg) + \frac{t^2}{2} \bigg( \Phi_2 + n_2 \partial_x  \bigg) + \Oc(t^3).
\end{equation}
We now want to impose that the above is a vector normal to $\gS$,
\begin{align}\label{3 - ProofLemma(normal bundle)1}
g\left(\Phi , \partial_s \right)=0 & & g\left(\Phi , \partial_t \right)=0.
\end{align}
We do so by solving the above equation order by order. With the expansion 
\begin{equation}\label{3- Lemma(Normal bundle)2}
\bar{g}\left(x(s,t), y^i(s,t)\right) = dx^2 + h + t^2\left.\left(\frac{1}{2}v^i \partial_i h - |\dg|^2 P\right) \right|_{t=0} + \Oc(t^3).
\end{equation}
of the restriction $g\big|_{\gS}$ of the metric to $\gS$ this is a direct computation and results in the expansion \eqref{3- Normal vector fields}.

If we now take $\phi(t) = \eta_a$ in \eqref{3- Normal vector fields}, all that is left to prove Lemma~\ref{3- Lemma: Normal vector fields bis} is prove that such vector fields form a basis of the normal bundle at every point $p \in \gS$. Using the expansion of the metric \eqref{3- Lemma(Normal bundle)2}, this is however straightforward to check that their inner product is \eqref{3- Normal vector field, inner product}. In particular, since this inner-product is non-degenerate, it proves their independence.
\end{proof}
	
\begin{Lemma}\label{3- Lemma: Extrinsic curvature}
Let $\g$ be a curve in $X$ and let $\gS$ be any surface with asymptotic boundary $\g$ in the parametrisation \eqref{3- Asymptotic surface, expansion} of Lemma~\ref{3- Lemma: Conformal parametrisation}. Let $\Phi \in N_{\gS}$ be an element of the normal bundle parametrised as in Equation \eqref{3- Normal vector fields} of Lemma~\ref{3- Lemma: Normal vector fields bis}. 

Then the extrinsic curvature has, in the coordinate basis $\{\partial_s, \partial_t\}$ an expansion of the form
\begin{equation}\label{3- Extrinsic curvature, general expansion}
 K\left(\Phi\right) = \frac{1}{t^3} \bigg[  0 + t\; K_1 +\frac{t^2}{2}\; K_2 + \Oc(t^3)\bigg]
\end{equation}
with
\begin{adjustwidth}{30pt}{0pt}
\begin{align*}
K_1 &=  
\begin{pmatrix}
 h\left(\phi(t), \nabla_{\dg}\left(|\dg|^{-2} \dg \right) -v\right)& 0 \\ 0 & 0
 \end{pmatrix}
\\ \\
K_2 &= 
\begin{pmatrix}
- h\left(\phi(t), n\right) |\dg|^{-2} & h\left(\phi(t),  	|\dg|^{-1}\nabla_{\dg}\left(|\dg|v\right) + \frac{1}{2}|v|^2 \dg - h^{-1}\left( P(\dg) \right)\right) 
\\ \\h\left(\phi(t),  	|\dg|^{-1}\nabla_{\dg}\left(|\dg|v\right) + \frac{1}{2}|v|^2 \dg - h^{-1}\left( P(\dg) \right)\right)
 &  h\left(\phi(t), n\right) |\dg|^{-2}
\end{pmatrix}\\
\end{align*}
\end{adjustwidth}
\end{Lemma}	
Note that, if we assume $\phi(0) \neq 0$, the normal $\Phi$ and coordinate vectors $\{\partial_s, \partial_t\}$ all have norms diverging as one over $t$, see \eqref{3- Induced metric} and \eqref{3- Normal vector field, inner product}.

\begin{proof}\mbox{}\\
This is again just an order by order computation of the different components of the extrinsic curvature,
\begin{align}
K\left(\partial_s , \partial_s , \Phi \right) &= g\left( \nabla_s \partial_s , \Phi\right), \nonumber\\
K\left(\partial_t , \partial_t , \Phi \right) &= g\left( \nabla_t \partial_t , \Phi\right), \\
K\left(\partial_s , \partial_t , \Phi \right) &= g\left( \nabla_s \partial_t , \Phi\right). \nonumber
\end{align}
From the expansion \eqref{3- FG Third Normal Form2} of $g$ one can derive the expansion of the Christoffel symbols
\begin{align*}
\gG_{xxx}&= \frac{1}{t^3 |\dg|^2}\left[ - \frac{1}{|\dg|} + t^2 \frac{x_3}{|\dg|^2} + \Oc(t^3)\right], & \gG_{ixj} &= \frac{1}{t^3 |\dg|^2}\left[ \frac{h_{ij}}{ |\dg|} + t^2\left( -\frac{x_3}{|\dg|^2}h + \frac{|\dg|}{2} v^i \partial_i h \right) + \Oc(t^3)\right], \nonumber\\
\gG_{xix}&= 0, & \gG_{xij} &= \frac{1}{t^3 |\dg|^2}\left[ -\frac{h_{ij}}{ |\dg|} + t^2\left( -\frac{x_3}{|\dg|^2}h - \frac{|\dg|}{2} v^i \partial_i h \right) + \Oc(t^3)\right], \\
\gG_{xxi}&= 0, & \gG_{ikj} &= \frac{1}{t^3 |\dg|^2}\left[ - t \gG_h{}_{ikj} + \Oc(t^3)\right], \nonumber
\end{align*}
with $x_3 = \bigg( -\frac{3}{4}|\dg|^3|v|^2 + \frac{1}{2} |\dg| \kappa(\g, v, h) \bigg)$. Then, since we already have the expansions \eqref{3- Lemma(Normal bundle)2} and \eqref{3- Normal vector fields} for $g\big|_{\gS}$ and the normal vectors fields, it is tedious but straightforward to derive the expansion \eqref{3- Extrinsic curvature, general expansion}.
\end{proof}

Combining Lemma~\ref{3- Lemma: Extrinsic curvature} with previous results we obtain the following corollaries:
\begin{Corollary}\label{3- Crllr: mean curvature}
Let $\g$ be a curve in $X$ and let $\gS$ be any surface with asymptotic boundary $\g$ in the parametrisation \eqref{3- Asymptotic surface, expansion} of Lemma~\ref{3- Lemma: Conformal parametrisation}. Let $\Phi \in N_{\gS}$ be an element of the normal bundle parametrised as in Equation~\eqref{3- Normal vector fields} of Lemma~\ref{3- Lemma: Normal vector fields}.
 Then the mean curvature of $\gS$ is
\begin{equation}
TrK(\Phi) = \frac{1}{t}\bigg( 0 + t \; h\left(\phi(t), \nabla_{\dg}\left(|\dg|^{-2} \dg \right) -v\right) + \Oc(t^3) \bigg).
\end{equation}
In particular, $\gS$ is asymptotically minimal if and only if $v$ satisfies the first conformal geodesic Equation \eqref{1- Eq: First Conformal Geodesic Eq}.
\end{Corollary}
\begin{proof}
This is a direct consequence of Lemma~\ref{3- Lemma: Extrinsic curvature} and Corollary~\ref{3- Crllr: Induced metric}. 
\end{proof}

\begin{Corollary}\label{3- Crllr: extrinsic curvature}
	Let $\g$ be a curve in $X$ and let $\gS_{\g}$ be its ambient surface given by Theorem~\ref{3- Thrm: Holographic dual}. Let $\Phi \in N_{\gS}$ be an element of the normal bundle parametrised as in Equation \eqref{3- Normal vector fields} of Lemma~\ref{3- Lemma: Normal vector fields}.	
Then the extrinsic curvature has, in the coordinate basis $\{\partial_s, \partial_t\}$, an expansion of the form
\begin{align}\label{3- Extrinsic curvature of the holographic dual}
&\; K\left(\Phi\right) = \frac{1}{t^3} \bigg[  0
+ t^2 
\begin{pmatrix}
0& 1 
\\ 1
&  0
\end{pmatrix} h\left(\phi(t) \;,\;  	|\dg|^{-1}\nabla_{\dg}\left(|\dg|v\right) + \frac{1}{2}|v|^2 \dg - h^{-1}\left( P_h(\dg) \right)\right)
+ \Oc(t^3) \bigg]
\end{align}
In particular, $\gS_{\g}$ is asymptotically totally geodesic if and only if $v$ satisfies the second conformal geodesic equation \eqref{1- Eq: Second Conformal Geodesic Eq} or equivalently if and only if $\g$ satisfies the conformal geodesic equation \eqref{1- Conformal Geodesic Equation}.
\end{Corollary}
\begin{proof}
This is a direct consequence of Lemma~\ref{3- Lemma: Extrinsic curvature} with Theorem~\ref{3- Thrm: Holographic dual}. 
\end{proof}

We now prove Theorem~\ref{3- Thrm: Conformal Geodesics}. 
\begin{proof}\mbox{}\\
Making use of Corollary~\ref{3- Crllr: extrinsic curvature} and the explicit metrics \eqref{3- Normal vector field, inner product} and \eqref{3- Induced metric} of Lemma~\ref{3- Lemma: Normal vector fields bis} and Corollary~\ref{3- Crllr: Induced metric} we can directly compute the norm of the extrinsic curvature \eqref{3- Extrinsic curvature of the holographic dual},
\begin{equation}
|K|^2 = 2\;t^4\; \bigg| |\dg|^2\; \left( |\dg|^{-1}\nabla_{\dg}\left(|\dg|v\right) + \frac{1}{2}|v|^2 \dg - h^{-1}\left( P_h(\dg) \right) \right)^{\perp} \bigg|^2 + \Oc(t^5)
\end{equation}	
where $\perp$ signals orthogonal projection in the direction normal to $\gS_{\g}$.

Thus generically $|K| = 0 + \Oc(t^2)$ and $|K| = 0 + \Oc(t^3)$ if and only if $v$ satisfies the normal part of the second conformal geodesic equation \eqref{1- Eq: Second Conformal Geodesic Eq}. Finally we recall that the first conformal geodesic equation is already satisfied as a result of Theorem ~\ref{3- Thrm: Holographic dual} and that the vanishing of the tangential part of the second conformal geodesic equations amounts to a choice of parametrisation (see section \ref{subsec: Conformal geodesics} and Theorem~\ref{3- Thrm: Conformal Parametrisation}). Altogether this proves Theorem~\ref{3- Thrm: Conformal Geodesics}.
\end{proof}

	\section{Ambient surface and tractors} \label{Sec: Ambient Metric Realisation of Conformal Geodesics and its Relation to Tractors}
	
	In section \ref{subsec: Dual surface to a conformal curve} we associated to any curve $\g$ in a conformal manifold $X$ a unique surface $\gS_{\g}$ in the related Poincar\'e--Einstein space $\left(Y, g\right)$. This surface can be taken to be asymptotically defined by the expansion \eqref{3- Holographic dual, expansion} of Theorem~\ref{3- Thrm: Holographic dual}. We then showed (see Theorem~\ref{3- Thrm: Conformal Geodesics}) that $\g$ is a conformal geodesic if and only if the extrinsic curvature of $\gS_{\g}$ vanishes to one order higher than generically expected.
	
	Here we rephrase these results in terms of the ambient ((n+2)-dimensional) metric of Fefferman and Graham \cite{fefferman_ambient_2012}. This has two main advantages: first it will explicitly relate our constructions to the standard tractor formulation of the conformal geodesic equations~\cite{bailey_thomass_1994};  second the use of tractor calculus will dramatically shorten the proofs so that this section can be seen as alternative and more transparent derivations of Theorem~\ref{3- Thrm: Conformal Geodesics}. The reason why we postponed this discussion is that it relies on the simultaneous use of tractor calculus and the ambient metric which are perhaps not as well-known tools as those used in section~\ref{Sec: Holographic Prescription for Conformal Geodesics}.

	\subsection{The ambient metric and tractors: a very brief overview}
	
		In this section we briefly review from \cite{cap_standard_2003} the relation between the ambient metric and tractor calculus but otherwise assume that the reader is familiar with standard tractor calculus, see \cite{bailey_thomass_1994,curry_introduction_2018}. We also review how the conformal geodesic equations can be elegantly written in terms of tractors.
	
	\subsubsection{The ambient metric associated to $\left(X, [h]\right)$ and its Fefferman--Graham expansion}\label{subsubsec: The ambient metric}
	
Let $\left(X , [h]\right)$ be an $n$-dimensional conformal manifold. Let $\left(M , g_M\right)$ be the associated ambient (formal) metric \cite{fefferman_conformal_1985}, \cite{fefferman_ambient_2012}: By construction, $M$ is $(n+2)$-dimensional with $M= (-1,1) \times L$ where $L$ is the total space of the scale bundle $L\to X$. Recall that a conformal metric $\hConf$ is a section of $L^2 \otimes S^2T^*X$. Let $\left(\gs, y\right)$ be coordinates on $L$, then the pullback $g_L=\pi^*\hConf = \gs^2 h$ of $\hConf$ gives a metric on $L$. Let $\gr\from M \to \R$ be a boundary defining function, $L = \gr^{-1}(0)$, $d\gr \neq 0$. The ambient metric $g_M$ on $M$ extends $g_L$ and has an asymptotic expansion
\begin{equation}\label{4- FG First Normal Form3}
		g_M= -2 \gr d\gs^2 -2 \gs d\gs d\gr +\gs^2 \left[h -2\gr P +\Oc(\gr^2)\right].
\end{equation}
As in previous sections, if $n>2$ then $P$ is the Schouten tensor of $h$ while if $n=2$, $P$ corresponds to choice of M\"obius structure. Since we won't need the precise form of the expansion beyond second order, Equation \eqref{4- FG First Normal Form3} can be taken to define the metric $g_M$ (this amounts to only requiring $g_M$ to be Ricci-flat to lowest order).

The ambient metric \eqref{4- FG First Normal Form3} can be related to the Poincar\'e--Einstein metric by making use of the change of coordinate
	\begin{equation}\label{4- Coordinates change: Ambient to Poincare}
	\gs = l x^{-1},\quad  \gr= \frac{1}{2}x^2.
	\end{equation}
We then have
	\begin{equation}\label{4- FG Second Normal Form2}
	g_M = -dl^2 + \frac{l^2}{x^2} \left(dx^2 + \left[ h - x^2 P + \Oc(x^4) \right] \right)
	\end{equation}		
and the hypersurface $l=1$ identifies with the Poincaré--Einstein manifold $\left(Y , g\right)$. In particular $Y$ is a graph over the bundle of scale $L$ defined by
 \begin{equation}
\gr = \frac{1}{2}\left( \gs \right)^{-2}
\end{equation}
(see figure \ref{fig:AmbientMetric}).

\begin{figure}[h]
	\begin{center}
	\includegraphics[scale=0.5]{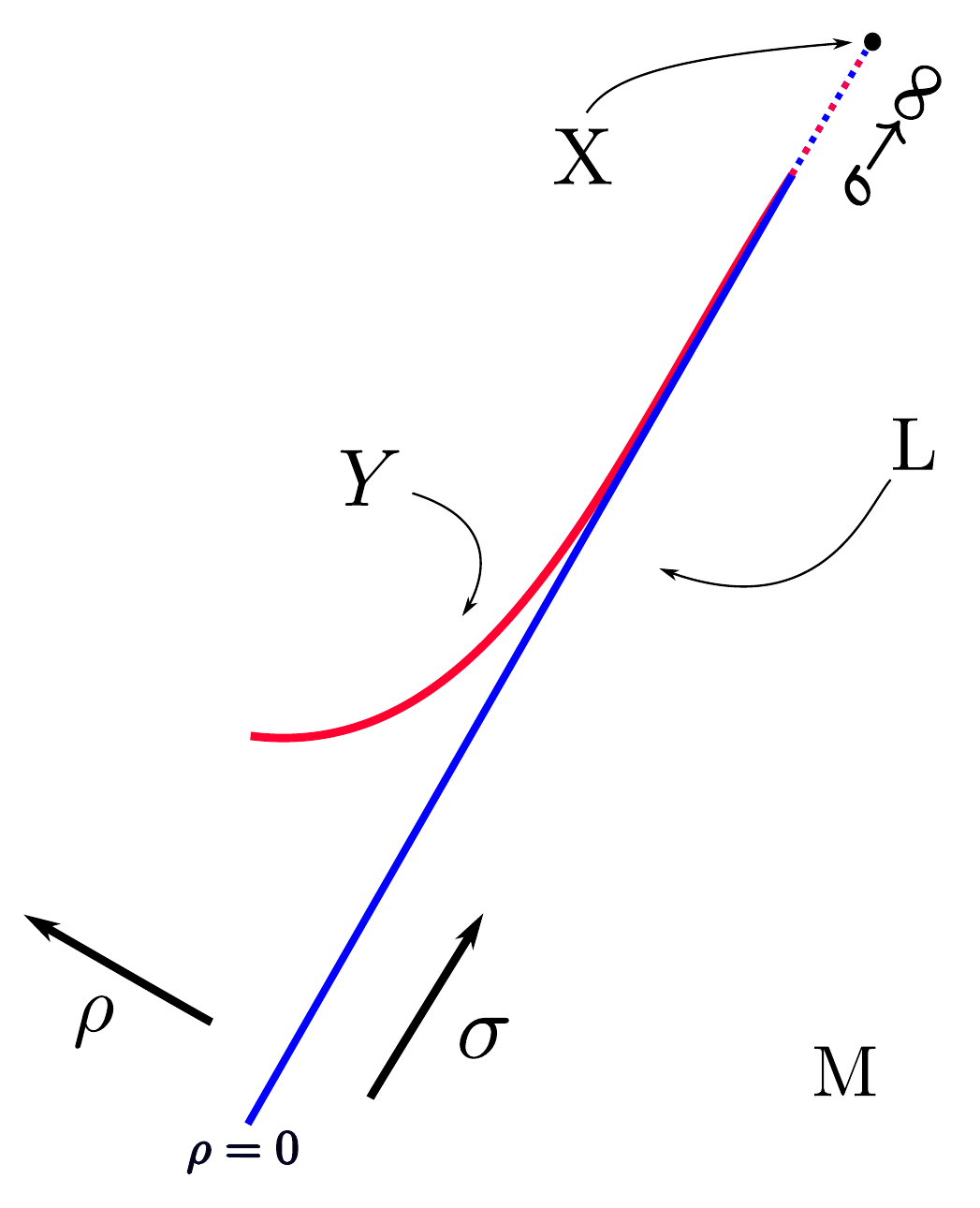}
	\end{center}
	  \caption{The Poincar\'e-Einstein ($Y$) manifold associated to $(X, \hConf)$ is a sub-manifold of the Ambient manifold ($M$). It is a graph over the bundle of scale $L \to X$. As $\gs$ goes to infinity, the conformal boundary $X$ of $Y$ identifies with a section of $L$.}
	  \label{fig:AmbientMetric}
\end{figure}

As $\gs$ goes to infinity the bundle of scale $L$ and the Poincar\'e--Einstein metric $Y$ both converge toward the same asymptotic boundary (identified with $X$). This somewhat fuzzy statement can be made precise by introducing a ``conformal compactification'' of the ambient metric generalizing Penrose's conformal compactification of Minkowski space. Since the ambient metric is only defined in a neighbourhood of $L$ there is only one copy of $X$ (``at null-infinity of $L$'') that needs to be added as a boundary (instead of a full asymptotic null cone for Minkowski compactification).

Finally, we note for future reference that the unit normal to $Y$ in $M$ is
\begin{equation}
\partial_l\Big|_{l=1} = \gs \partial_{\gs}\Big|_{l=1} =: E\Big|_{l=1}.
\end{equation}
(Here and everywhere in the subsequent subsections we will make use of the notation $E := \gs \partial_{\gs}$).
	
\subsubsection{Tractors and the ambient metric}	\label{subsubsec: Tractors and the ambient metric}

We briefly review from \cite{cap_standard_2003,gover_conformally_2003} the close relationship between tractor calculus on a conformal manifold $\left(X, \hConf\right)$ and the associated ambient metric $\left(M,g_M\right)$.

As we just discussed, the bundle of scales $L\to M$ can be seen as a hyper-surface in the ambient manifold $\left(M, g_M\right)$ through the ambient metric construction \cite{fefferman_conformal_1985},\cite{fefferman_ambient_2012}. The bundle of scales is a $\R^+$-principal bundle and the restriction $g_M \big|_{L}$ is homogeneous of degree $2$ under this action. Now let $y \in X$ and $\pi^{-1}(y) \subset L$ be the fibre over $y$ in $L$. Let $\iota \from \pi^{-1}(y) \to M$ be its inclusion in $M$ and consider sections of the pull-back bundle $\gG\left[\iota^* TM\right]$ which are homogeneous of degree $-1$ under the $\R^+$-action. This defines a $n+2$-dimensional vector space $\mathcal{T}_y$. Since the contraction of such sections with $g_M \big|_{L}$ is homogeneous degree zero this gives a metric on $\mathcal{T}_y$. Altogether this defines a $n+2$-dimensional metric vector bundle $\mathcal{T} \to X$ over $X$. It was shown in \cite{cap_standard_2003} that it canonically identifies with the standard tractor bundle. 
By construction, if $\iota \from L \to M$ is the inclusion, sections of $\mathcal{T}$ are sections of $\iota^* TM$ which are homogeneous degree $-1$ under the $\R^+$-action.
 
 It was also shown in \cite{cap_standard_2003} that the Levi-Civita connection $\nabla$ of the ambient metric induces a connection on the tractor bundle, as follows: 
 
   Let $Y^I$ be a section of the tractor bundle $Y^I \in \gG\left[\mathcal{T}\right]$, by construction this is a section of the pull-back bundle $\iota^* TM$ homogeneous of degree $-1$ under the $\R^+$-action. One first proves, see \cite{cap_standard_2003}, that the Levi-Civita connection $\nabla$ of the ambient metric satisfies,
 \begin{equation}\label{4- Identity: ambient LC1}
 \nabla Y^I = \nabla\Big|_{\rho=0} Y^I + \Oc(\rho) Y^I
 \end{equation}
 and
 \begin{align}\label{4- Identity: ambient LC2}
 \nabla_E\Big|_{\rho=0} Y^I &=0, & \nabla_Y\Big|_{\rho=0} E = Y^I
 \end{align} 
 (recall that $E = \gs \partial_{\gs}$).  Let $y \in\gG\left[TX\right]$ and let $\pi^{-1}(y) \in \gG\left[TL\right]$ be any lift of $y$ to $L$. From \eqref{4- Identity: ambient LC2} one sees that $\nabla_{\pi^{-1}(y)}\Big|_{\rho=0} Y^I$ does not depend on the choice of lift. One then show that $\nabla_{\pi^{-1}(y)}\Big|_{\rho=0} Y^I \in \gG\left[\iota^*TM\right]$ is homogeneous degree $-1$ under the $\R^+$-action and thus defines a section of $\mathcal{T}$. The Levi-Civita connection of the ambient metric thus induces a connection $D \from \mathcal{T} \to T^*X \otimes \mathcal{T}$ on the tractor bundle as
\begin{equation}
D_y Y^I \coloneqq \nabla_{\pi^{-1}(y)}\Big|_{\rho=0} Y^I.
\end{equation} It was proved in \cite{cap_standard_2003} that this connection actually coincides with the normal tractor connection. Finally, the tractor metric $g_{IJ} \in \Gamma\left[S^2 T^* \mathcal{T}\right]$ is obtained by taking inner product of tractors (seen as homogeneous degree $-1$ sections of $\iota^* TM \to L$) with the ambient metric (itself homogeneous degree $2$). This makes sense since the resulting function on $L$ is homogeneous degree zero, i.e a function on $X$.
	
\subsubsection{Tractor calculus and conformal geodesics}\label{subsubsec: Tractor calculus and conformal geodesics}

	We here review the tractor formulation of the conformal geodesic equations, see \cite{bailey_thomass_1994} for the original exposition and \cite{gover_distinguished_2018,gover_distinguished_2020} for recent developments.
	
	Let $\left(X , [h] \right)$ be an $n$-dimensional conformal manifold. We take $\gamma \from I \to X$ to be a curve in $X$ parametrised by a parameter $s$ and $\dg = \partial_s \gamma$. Let $h$ be a choice of representative for $\hConf$ and $|\dg|^2 = h\left(\dg, \dg \right)$. 
	
	Define the ``position tractor'' to be,
	\begin{equation}
	|\dg|^{-1} X^{I} = \begin{pmatrix}
	0 \\ 0 \\  \; |\dg|^{-1}
	\end{pmatrix}.
	\end{equation}
	By subsequent differentiations of the ``position tractor'' one obtains the ``velocity tractor'',
	\begin{equation}\label{4- Velocity Tractor}
	U^{I}= D_{\dg} \left(|\dg|^{-1}  X^{I} \right)= \begin{pmatrix}
	0 \\ |\dg|^{-1}\;\dg \\  \; \partial_s\left( |\dg|^{-1} \right)
	\end{pmatrix},
	\end{equation}
	and ``acceleration tractor'' :
	\begin{equation}
	A^{I}= D_{\dg} U^{I} = \begin{pmatrix}
	|\dg| \\ |\dg|v  \\  |\dg| \frac{1}{2} |v|^2 +  |\dg|^{-1}  \kappa(\g,v, h)
	\end{pmatrix}
	\end{equation}
	where
	\begin{align}
	v &\coloneqq \nabla_{\dg}\left(|\dg|^{-2} \dg \right), &
	\kappa(\g,v, h) &\coloneqq h\left(\dg , |\dg|^{-1}\nabla_{\dg}\left(|\dg|v\right) + \frac{1}{2}|v|^2 \dg - h^{-1}\left( P_h(\dg) \right)  \right).
	\end{align}
		
	It was proved in \cite{bailey_thomass_1994} that the vanishing of $A^{I}A_{I}= -2 \kappa(\g,v, h) $ is equivalent to the preferred conformal parametrisation of $\gamma$ while $D_{\dg}A^{I}=0$ is equivalent to the (parametrised) conformal geodesic equations \eqref{1- Conformal Geodesic Equation}.

Making use of the identification of sections of the tractor bundle with (homogeneous degree -1) sections of $\iota^* TM$ we have
\begin{align}\label{4- Position and Velocity Tractor}
|\dg|^{-1} X^I\left(\gs , y\right) &= \gs^{-1} \Big( |\dg|^{-1} E \Big), &
U^I\left(\gs , y\right)  &= \gs^{-1}\Big( |\dg|^{-1} \dg + \partial_s \left(|\dg|^{-1}\right) E \Big),
\end{align}
and
\begin{align}\label{4- Acceleration Tractor}
A^{I}\left(\gs , y\right)  &= \gs^{-1} \Bigg(
|\dg| \partial_{\gr} + |\dg|v  +  \left(|\dg| \frac{1}{2} |v|^2  +|\dg|^{-1} \kappa(\g,v, h)\right) E
\Bigg).
\end{align}

\subsection{Tractor calculus for ambient surfaces of conformal geodesics}

	As we have reviewed, the tractor calculus has a direct interpretation in terms of the Ricci calculus of the ambient manifold $(M, g_M)$. On the other hand the Poincar\'e--Einstein manifold $\left(Y,g\right)$ can always be thought as a sub-manifold of the ambient one. The ambient space therefore a good setting for
	relating the conformal geodesic equations of Section \ref{Sec: Holographic Prescription for Conformal Geodesics} with their tractor
	counterparts. This is one of the aims of this subsection. The second aim is to provide an alternative, more compact, proof of Theorem \ref{3- Thrm: Conformal Geodesics}.

\subsubsection{Tractor calculus, ambient metric and conformal curves}

Let $\g \from I \to X$ a curve in $X$ parametrised by a parameter $s$. The bundle of scales $L\to X$ naturally is a $\R^+$-principal bundle and we can thus consider $\Sc_{\g} \from I^2 \to L$, the unique $\R^+$-invariant lift of $\g$ to $L$. It satisfies
\begin{align}
\pi\left(\Sc_{\g}\right) &= \g, &  R_k \Sc_{\g} = \Sc_{\g} \quad  \forall k \in \R^+.
\end{align}
Where $R_k \from L \to L$ is the diffeomorphism given by the action of $k \in \R^+$ on $L$.

In coordinates, a convenient parametrisation is
\begin{equation}\label{4- L surface parametrisation}
\Sc_{\g} = \left\{
\begin{array}{ccc}
\gr\left(s,t\right) & = & 0 \\ \\
y^i\left(s,t\right) & = & \g(s) \\ \\
\gs\left(s,t\right) & = & f(t)\;|\dg|^{-1} 
\end{array}
\right\}.
\end{equation}
Since the metric induced by \eqref{4- FG Second Normal Form2} on $L$ is $g_L = \gs^2 h$ the corresponding induced metric on $\Sc_{\g}$ is degenerate. The above parametrisation has the following property: for fixed value of $t$, the surface $\Sc_{\g}$ restricts to a lift of $\g$ in $L$ with constant length $f(t)$ (with respect to $g_L = \gs^2 h$). The precise value of the length is a remaining choice of parametrisation.

The coordinate vectors given by this parametrisation are
\begin{align}
\partial_s &= f(t) \;U^I(s,t), &
\partial_t & = f'(t)  \; |\dg|^{-1}X^I(s,t),
\end{align}
where $|\dg|^{-1} X^i(s,t)$ and $U^I(s,t)$ are the position and velocity tractors \eqref{4- Position and Velocity Tractor} evaluated at $(\gs = \gs(s,t), y = y(s,t))$.

The above notation is convenient as covariant differentiation along $s$ -with
respect to the Levi-Civita connection of the ambient metric- can then be rewritten in terms of the tractor connection:
\begin{align}
\nabla_s \partial_t &= f'(t) D_s\left(|\dg|^{-1} X^I\right) = f'(t) \;U^I(s,t), &
\nabla_s \partial_s &= f(t) D_s U^I = f(t) \;A^I(s,t).
\end{align}
where $A^I(s,t)$ is the acceleration tractor \eqref{4- Acceleration Tractor} evaluated at $(\gs = \gs(s,t), y = y(s,t))$. 
We also immediately have $\nabla_s \nabla_s \partial_s =f(t) D_s A^I$ and thus $\g$ is a conformal geodesic if and only if
 \begin{equation}
	\nabla_s \nabla_s \partial_s=0.
\end{equation}

This tractor interpretation of the covariant differentiation along the surface $\Sc_{\g}$ in $L$ considerably simplifies our task. We will see that essentially the same identification is possible for the ambient surface $\gS_{\g}$ in $Y$: as $Y$ and $L$ asymptotically converge to the same asymptotic boundary, differential calculus on $\gS_{\g}$ in $Y$ asymptotically identifies with tractor calculus.

From now one we will take as a convenient parametrisation for $\Sc_{\g}$:
\begin{equation}\label{4- f(t) =1/t }
f(t) = \frac{1}{t}.
\end{equation}
With this parametrisation, when $t$ goes to zero the surface $\Sc_{\g}$ goes to the asymptotic boundary of $L$ (which, we recall, is identified with $X$).

\subsubsection{Ambient surface in the ambient metric}

In section \ref{subsec: Dual surface to a conformal curve} we associated to any curve $\g$ in $X$ an ambient surface $\gS_{\g}$ in $Y$. The following proposition gives the interpretation of this result in terms of the ambient metric. In fact, since $Y$ -identified with the hypersurface $l=1$ in $M$- is a graph over $L$, it is reasonable to expect that $\gS_{\g} \subset Y$ is a graph over $\Sc_{\g}  \subset L$ which is indeed the case:

\begin{Proposition}\mbox{}\label{4- Thrm: Holographic dual}\\	
Let $\g$ be a curve in $X$ and let $\left(M, g_M\right)$ be the ambient space asymptotically defined by \eqref{4- FG First Normal Form3}. Let $\Sc_{\g}$ be the $\R$-invariant lift of $\g$ to $L$ with the parametrisation \eqref{4- L surface parametrisation}, \eqref{4- f(t) =1/t }. 

 The ambient surface $\gS_{\g}$ to $\g$ is asymptotically a distinguished graph over $\Sc_{\g}$ given by the following expansion,
 \begin{equation}\label{4- Ambient dual}
 	\gS_{\g} = \left\{
 	\begin{array}{cccccc}
 		\gr\left(s,t\right) & = & 0                &+ \frac{t}{2} A^{\gr}(s,t)   & + \Oc(t^4)\\ \\
 		y^i\left(s,t\right) & = & \g(s)            &+ \frac{t}{2} A^{i}(s,t)  & + \Oc(t^4)  \\ \\
 		\gs\left(s,t\right) & = & f(t)\;|\dg|^{-1} &+ \frac{t}{2} \left( A^{\gs}(s,t) + \frac{1}{3}\left(A_JA^J\right) \;|\dg|^{-1} \right) & + \Oc(t^3)
 	\end{array}
 	\right\}.
 \end{equation}
Where $(A^{\gr},A^{i},A^{\gs})$ are the coordinates of the acceleration tractor, thought as the vector field \eqref{4- Acceleration Tractor}.

What is more, upon identifying the Poincaré--Einstein manifold $\left(Y, g\right)$ with the hypersurface $l=1$, this definition coincides with the definition of Theorem~\ref{3- Thrm: Holographic dual}. In particular the parametrisations \eqref{3- Holographic dual, expansion} and \eqref{4- Ambient dual} coincide.
\end{Proposition}
Along the way, proposition \ref{4- Thrm: Holographic dual} gives us another simple geometrical interpretation of the preferred conformal parametrisation for $\g$:
$\gS_{\g}$ is a graph above $\Sc_{\g}$ given by the acceleration tractor if and only if $\g$ has been given a preferred conformal parametrisation $A_IA^I = -2 \kappa(\g,v, h)=0$.

\begin{proof}\mbox{}\\
We recall the expansion of $\gS_{\g}$ given by Theorem~\ref{3- Thrm: Holographic dual}:
\begin{equation}
\gS_{\g} =\left\{\begin{array}{ccccccccccccccc}
x\left(s,t\right) & =& 0 &+& t |\dg| &+& 0 &+&\frac{t^3 }{3} x_3 &+& \frac{t^4 }{4} x_4 &+&\Oc\left(t^5\right) \\ \\
y^i\left(s,t\right) & =& \g^i(s) &+& 0  &+& \frac{t^2}{2} |\dg|^2 v^i &+& 0 &+& \Oc\left(t^4\right)\\ \\
l(s,t) & = &1
\end{array} \right\}\nonumber
\end{equation}
with $-\frac{1}{3}|\dg|^{-2} x_3 = \frac{1}{4}|\dg||v|^2 - \frac{1}{6} |\dg|^{-1} \kappa(\g,v, h)$. One can also show that if $\gS_{\g}$ is given an isothermal parametrisation then $x_4=0$. (If $n=3$ this uses the fact that we took the (free) third order coefficient in the expansion \eqref{4- FG Second Normal Form2} to be zero. This is however not really a choice here since from \cite{fefferman_conformal_1985}, \cite{fefferman_ambient_2012} this is forced on us if we want the ambient metric to be smooth.)

 Making use of the change of coordinates \eqref{4- Coordinates change: Ambient to Poincare} one readily sees that it is equivalent to \eqref{4- Ambient dual}.
\end{proof}

\subsubsection{Extrinsic curvature in the ambient metric}\label{subsubsec: Extrinsic curvature in the ambient metric}

We now show how the identification of tractor calculus with the asymptotic tensor calculus on $\gS_{\g}$ in the ambient metric simplifies the calculation of the extrinsic curvature of $\gS_{\g}$.

The parametrisation \eqref{4- Ambient dual} induces the coordinate basis
\begin{align}\label{4- Coordinate basis}
	\partial_s &= \frac{1}{t} U^I +\Oc(t) X^I + \Oc(  t^2), \\
	\partial_t &= \Big(-\frac{1}{t} + \frac{t}{3} A_J A^J \Big)E + A^I  + \Oc(t^2) X^I + \Oc(t^3).\nonumber
\end{align}	
Implicitly, here and in what follows, all tractors are evaluated at the point $\gs(s,t)$ given by \eqref{4- Ambient dual}. 
	 
	 With the identities \eqref{4- Identity: ambient LC1}, \eqref{4- Identity: ambient LC2} in hand, the covariant derivatives of \eqref{4- Coordinate basis} are just given by a few lines of computation:
	 
	 \begin{flalign}
	 \nabla_s \partial_s &= \frac{1}{t} \nabla_s\Big|_{\rho=0} U^I +\Oc(t) X^I + \Oc(\gs^{-1}  t) = \frac{1}{t} D_s U^I +\Oc(t) X^I + \Oc(t^2) \\
	 &= \frac{1}{t} A^I +\Oc(t) X^I + \Oc(t^2) = \Big(\frac{1}{t^2} - \frac{1}{3} A_J A^J \Big)E + \frac{1}{t} \partial_t +\Oc(t) X^I + \Oc(t^2), \nonumber&
	 \end{flalign}
	 \vspace{0.1cm}	 	 
	 \begin{flalign}
	 \nabla_t \partial_t &=  \left(\frac{1}{t^2} + \frac{1}{3} A_J A^J \right)E + \Big(-\frac{1}{t} + \frac{t}{3} A_J A^J \Big) \nabla_{t}\Big|_{\rho=0} E + \nabla_{t}\Big|_{\rho=0} A^I + \Oc(t) X^I + \Oc(t^2) &\\ 
	 &= \left(\frac{1}{t^2} + \frac{1}{3} A_J A^J \right)E + \Big(-\frac{1}{t} + \frac{t}{3} A_J A^J \Big) \partial_t 
+ \nabla_{A}\Big|_{\rho=0} A^I + \Oc(t) X^I + \Oc(t^2)  \nonumber&\\
	 &= \left(\frac{1}{t^2} + \frac{1}{3} A_J A^J \right)E + \Big(-\frac{1}{t} + \frac{t}{3} A_J A^J \Big) \partial_t  + \Oc(t) X^I + \Oc(t^2), \nonumber&
	 \end{flalign}
\vspace{0.1cm}	 
	 \begin{flalign}
	 \nabla_s \partial_t &=  \frac{t}{3}\partial_s \left( A_J A^J \right)E + \Big(-\frac{1}{t} + \frac{t}{3} A_J A^J \Big) \nabla_{s}\Big|_{\rho=0} E + \nabla_{s}\Big|_{\rho=0} A^I + \Oc(t^2) X^I + \Oc(t^3) \\
	 &= \frac{t}{3}\partial_s \left( A_J A^J \right)E + \Big(-\frac{1}{t} + \frac{t}{3} A_J A^J \Big) \partial_t + D_s A^I + \Oc(t^2) X^I + \Oc(t^3).\nonumber&
	 \end{flalign}
To obtain the extrinsic curvature, all is left to do is project out the directions \eqref{4- Coordinate basis}, tangent to $\gS_{\g}$.
\begin{Proposition}\mbox{}\label{4- Prop: Covariant derivatives}\\
	Let $\g$ be a curve in $\left(X,[h]\right)$. Let $\gS_{\g}$ be its ambient surface in the ambient space $\left(M, g_{M}\right)$ in the parametrisation given by Proposition~\ref{4- Thrm: Holographic dual}. Then its extrinsic curvature in $M$ is
\begin{align}
K_M\left(\partial_s, \partial_s \right) &= \Big(\frac{1}{t^2} - \frac{1}{3} A_J A^J \Big)E +\Oc(t^2) E + \Oc(t^2), \\
K_M\left(\partial_t, \partial_t \right) &= \left(\frac{1}{t^2} + \frac{1}{3} A_J A^J \right)E  + \Oc(t^2) E + \Oc(t^2), \\
K_M\left(\partial_s, \partial_t \right) &=  D_s A^I +
 \left( A_J A^J\right) U^I+ \Oc(t^2)E + \Oc(t^2).\label{4- Covariant derivatives}
\end{align}
\end{Proposition}
Note that $\Oc(t^2)E = \Oc(t) X^I$. Since $E$ is the unit normal to $Y$ in $M$ we immediately have a straightforward alternative proof to Theorem~\ref{3- Thrm: Conformal Geodesics}:

Suppose that $\g$ satisfies the conformal geodesic equations $D_s A^I=0$ then the extrinsic curvature of $\gS_{\g}$ in $M$ in the coordinate basis is proportional to $E$ up to order $t^2$ and therefore the extrinsic curvature of $\gS_{\g}$ in $Y$ vanishes up to order $t^2$. Since the norm of the coordinate basis is of order one over $t$ and the Poincaré metric diverges as $t^{-2}$ the norm of the extrinsic curvature vanishes up to order $t^2$. 

The other way round, if the norm of the extrinsic curvature vanishes up to order $t^2$ then, in the coordinate basis, the extrinsic curvature of $\gS_{\g}$ in $Y$ vanishes up to order $t^2$. Consequently the extrinsic curvature of $\gS_{\g}$ in $M$ is proportional to $E$ up to order $t^2$ and from Proposition~\ref{4- Prop: Covariant derivatives} this implies that $D_s A^I =0$ which is equivalent to the conformal geodesics equations.

{
\bibliography{Biblio}
}

\end{document}